\documentclass[12pt]{article}
\usepackage{CJK}
\usepackage{amsmath}
\usepackage{amsthm}
\usepackage{amssymb}
\usepackage{amstext}
\usepackage{mathrsfs}
\usepackage{epic}
\usepackage{eepic}
\usepackage{graphicx}
\usepackage{picins}
\usepackage{subfigure}
\usepackage[center]{caption2}
\usepackage[bf,big]{titlesec}
\usepackage{array}
\newtheorem{definition}{Definition}[subsection]

\newtheorem{proposition}[definition]{Proposition}

\newtheorem{lemma}[definition]{Lemma}
\newtheorem{theorem}[definition]{Theorem}
\theoremstyle{definition}

\usepackage[numbers,sort&compress]{natbib}

\renewenvironment{proof}{
\noindent{\bf Proof.}\rm} {\mbox{}\hfill\rule{0.5em}{0.809em}\par}

\usepackage{extarrows}
\usepackage{pifont}
\usepackage{mathrsfs}
\usepackage{leftidx}
\begin{document}
\title{Automaticity of one-relator semigroups with length less than or equal to three\footnote{Supported by the NNSF of China (11571121) and the Science and Technology Program of Guangzhou (201605121833161).}}
\author{ Yuqun Chen\footnote {Corresponding author.}, Haibin Wu, Honglian Xie \\
{\small  School of Mathematical Sciences, South China Normal
University}\\
{\small Guangzhou 510631, P. R. China}\\
{\small yqchen@scnu.edu.cn}\\
{\small 378238979@qq.com}\\
{\small honglianxie@126.com}}

\date{}

\maketitle \noindent\textbf{Abstract:}
The main results of this paper is to give a complete characterization of the automaticity of one-relator semigroups with length less than or equal to three.
Let $S=sgp\langle A|u=v\rangle$ be a semigroup generated by a set $A=\{a_1,a_2,\dots,a_n\},\ n\in \mathbb{N}$ with defining relation $u=v$, where $u,v\in A^*$ and $A^*$ is the free monoid generated by $A$. Such a semigroup is called a one-relator semigroup. Suppose that $|v|\leq|u|\leq3$, where $|u|$ is the length of the word $u$. Suppose that $a,b\in A,\ a\neq b$. Then we have the following:
(1) $S$ is  prefix-automatic if $u=v\not\in \{aba=ba,\ aab=ba,\ abb=bb\}$. Moreover, if $u=v\in \{aba=ba,\ aab=ba,\ abb=bb\}$ then $S$ is not automatic.  (2) $S$ is biautomatic if one of the following holds: (i) $|u|=3,\ |v|=0$, (ii) $|u|=|v|=3$, (iii) $|u|=2$ and $u=v\not\in \{ab=a,\ ab=b\}$. Moreover, if $u=v\in \{ab=a,\ ab=b\}$ then $S$ is not biautomatic.

\ \


\section{Introduction}\label{Intro}

\qquad The research of automaticity of groups started in the 1980's. Many scholars engaged in the study of this field and found a lot of research production. For instance, \cite{aga,wp,sc,sgag,aita}.

In the end of 1990's, the concept of automaticity was generalized to semigroups and monoids. In papers \cite{dpas,as,ac,sa,ascs}, the authors established the basic theory and obtained some results about automatic semigroups.

Not all properties of automatic groups hold for automatic semigroups. For example, the (first order) Dehn function of an automatic group is bounded above by a quadratic function but the (first order) Dehn function of an automatic semigroup may not be bounded above by any primitive recursive function, see \cite{sgag}. A semigroup is called left-left (left-right, right-left,  right-right, resp.) automatic semigroup if there exists $(A, L)$ such that for any $a\in A\cup\{\varepsilon\},\ _a^\$L\ (\leftidx{^\$}{L}_a,\ \leftidx{_a}{L}^\$,\  L_a^\$,\ resp.$) (see Definition \ref{bia}) is regular. Examples show that the four types of automatic semigroups are not equivalent, however, the four types of automatic groups are equivalent. An automatic group can be characterized by a geometric property of their Cayley graph, which is intuitively the following: ``there is a constant $c$ such that, if two fellows travel at the same speed by two paths ending at most one edge apart, then the distance between is always less than $c$", see \cite{wp}. This property, called  the fellow traveller property, plays an essential role in many of the results obtained so far about automatic groups, but does not characterize automatic semigroups, see \cite{as}. Therefore, the geometric theory that holds for automatic groups does not hold for automatic semigroups. However, Hoffmann and Thomas \cite{agcoas} introduced the definition of directed fellow traveller property. They proved that if $M$ is a left-cancellative monoid finitely generated by $A$ and if $L$ is a regular subset of the set of all words over $A$ that
maps onto $M$, then $M$ is automatic if it has the directed fellow traveller property.

Many semigroups, monoids and groups have been proved automatic, such as free groups, free semigroups, braid groups, braid monoids (\cite{wp}), divisibility monoids (\cite{ft}), plactic monoids (\cite{pl}) and Chinese monoids (\cite{bsfch}), and so on.

In recent years, research of the automaticity of semigroups is active, for instance, \cite{pl,bsfch,dpas,as,ac,sa,ascs,amcg,baxv,ABC,bisg,agcoas,rrs,fdfoam}. Automatcity theory of groups and semigroups have become important in today's computer algebra.

Suppose semigroup $S$ is generated by a finite set $A$ and $A^+$ is the free semigroup (without identity) generated by $A$. The key to decide whether a semigroup $S$ is automatic is to find a regular language $L\subseteq A^+$ such that $(A,L)$ is an automatic structure of $S$. We usually choose $L$ to be a normal form of $S$. Generally, it is not easy to find a normal form of a semigroup. However, Gr\"{o}bner-Shirshov bases theory can help us to solve this problem.

Soppose $S=sgp\langle A|R\rangle$ to be a monoid generated by a set $A$ with defining relations $R$. If the cardinal number of $R$ is 1, we call $S$ a one-relator semigroup. One-relator semigroups is a kind of important semigroups and whether the word problem of one-relator semigroups is solvable is still  open.

The paper is organized as follows. In section 2, we review the concepts of regular language, automatic semigroup, biautomatic semigroup, prefix-automatic semigroup, and  Gr\"{o}bner-Shirshov bases for associative algebras. We cite also some results that will be used in the paper. In section 3, we give some characterizations of some automatic (biautomatic, prefix-automatic) semigroups, in particular, a complete characterization of the automaticity of one-relator semigroups with length less than or equal to three is given.

\section{Preliminaries}

In this section, we give some definitions, notions and mention some results that will be used in the paper.

\subsection{Regular language}

\qquad For any  set $A$, we denote $A^+$ to be the set of all non-empty words over $A$ and $A^*$ to be the set of all words over $A$ including the empty word $\varepsilon$. If $A$ is a generating set of a semigroup $S$, interpreting concatenation is multiplication in $S$. We induce a semigroup epimorphism $\phi:\,A^+\rightarrow S$.  For convenience, we write $\alpha$ as element $\phi(\alpha)$ in $S$, and write $\alpha\equiv\beta$ when $\alpha,\beta$ are the same element in $A^*$, write $\alpha=\beta$ when $\alpha$ and $\beta$ represent the same element of $S$.

We also say that $A$ is an alphabet and call language any subset of $A^*$. We will consider regular language, i.e. those languages accepted by finite state automata, see \cite{wp}, for example. For any words $\alpha,\beta\in A^*$, letter $a\in A$, regular language $L$ over $A$, and semigroup $S$, we denote
\begin{align*}
|\alpha|&:\ the\ length\ of\ \alpha\ (|\varepsilon|=0),\\
M(L)&:\ the\ finite\ state\ automaton\ accepting\ L,\\
|S(M(L))|&:\ the\ number\ of\ states\ in\ M(L),\\
 DFSA&:\ deterministic\ finte\ state\ automaton,\\
occ(a,\alpha)&:\ the\ number\ of\ occurrences\ of\ a\ in\ \alpha,\ for\ example,\ if\ \alpha\equiv\ abababb, \\
&\ \ \ then\ occ(a,\alpha)=3,occ(b,\alpha)=4,\\
con(\alpha)&:\ the\ set\ of\ letters\ that\ occurred\ in\ \alpha,\ for\ example,\ if\ \alpha\equiv abababbc, \\
&\ \ \ then\ con(\alpha)=\{a,b,c\},\\
\mathcal{P}(C)&:\ the\ set\ of\ subsets\ of\ set\ C,\\
Pref(L)&:\ the\ set\ of\ prefix\ words\ of\ words\ in\ L,\\
S^1&:\ the\ monoid\ formed\ by\ adding\ an\ identity\ element\ to\ S.
\end{align*}

Suppose $\alpha\equiv a_1a_2\cdots a_n\in A^+$, where each $a_i\in A$. We define $\alpha(0):=\varepsilon$ and for any $t\geq1$,
\begin{flalign}
\quad\quad\ \ \  & \alpha(t):=
  \begin{cases}
    {a_1a_2\cdots a_t} & \mbox{ if $t\leq n$},\\
    {a_1a_2\cdots a_n} & \mbox{ if $t>n$},
  \end{cases}  & \nonumber\\
\quad\quad\ \ \  & \alpha[t]:=
  \begin{cases}
    {a_{n-t+1}a_{n-t+2}\cdots a_n} & \mbox{ if $t\leq n$},\\
    {a_1a_2\cdots a_n} & \mbox{ if $t>n$}.
  \end{cases}  & \nonumber
\end{flalign}

A $gsm$ (generalized sequential machine) is a six-tuple $\mathcal{A}=(Q,A,B,\mu,q_0,T)$ where $Q,A$ and $B$ are finite sets (called the states, the input alphabet and the output alphabet, resp.), $\mu$ is a (partial) function from $Q\times A$ to finite subsets of $Q\times B^*$, $q_0\in Q$ is the initial state and $T\subseteq Q$ is the set of terminal states. The inclusion $(q',u)\in \mu(q,a)$ corresponds to the following situation: if $\mathcal{A}$ is in state $q$ and reads input $a$, then it can move into state $q'$ and output $u$.
We can interpret $\mathcal{A}$ as a directed labelled graph with vertices $Q$, and an edge $q\xlongrightarrow{(a,u)} q'$ for every pair $(q',u)\in \mu(q,a)$. For a path
$$\pi\ :\ q_1\xlongrightarrow{(a_1,u_1)}q_2\xlongrightarrow{(a_2,u_2)}q_3\cdots\xlongrightarrow{(a_n,u_n)}q_{n+1}$$
we define
$$
\Phi(\pi):=a_1a_2\cdots a_n,\ \ \  \Sigma(\pi):=u_1u_2\cdots u_n.
$$
For any $q,q'\in Q,u\in A^+$ and $v\in B^+$ we write $q\xlongrightarrow{(u,v)}_+q'$ to mean that there exists a path $\pi$ from $q$ to $q'$ such that $\Phi(\pi)\equiv u$ and $\Sigma(\pi)\equiv v$, and we say that $(u,v)$ is the $label$ of the path. We say that a path is successful if it has the form $q_0\xlongrightarrow{(u,v)}_+t$ with $t\in T$.

 Any $gsm$ $\mathcal{A}$ induces a mapping $\eta_{\mathcal{A}}$ : $\mathcal{P}(A^+)\xlongrightarrow{}{\mathcal{P}(B^+)}$ defined by
$$\eta_{\mathcal{A}}(X)=\{v\in B^+|(\exists u\in X)(\exists t\in T)(q_0\xlongrightarrow{(u,v)}_+t)\}.$$
An useful result is that if $X$ is regular then so is $\eta_{\mathcal{A}}(X)$, see \cite{ascs}.

\begin{definition}\label{pad}\rm{(\cite{as})} Let $A$ be an alphabet and $\$$ be a new symbol not
in $A$. Let $A(2,\$)=(A\cup
\{\$\})\times(A\cup\{\$\})-\{(\$, \$)\}$. Define the mapping $\delta_A^R: A^*\times A^{*}\rightarrow (A(2, \$))^{*}$ by
\begin{equation*}
(u_{1}\cdots u_{m},v_{1}\cdots v_{n})\mapsto
\begin{cases} (u_{1},v_{1})\cdots (u_{m},v_{n})& \text{if}\  m=n, \\
(u_{1},v_{1})\cdots (u_{n},v_{n})(u_{n+1},\$)\cdots (u_{m},\$)& \text{if}\  m>n, \\
(u_{1},v_{1})\cdots (u_{m},v_{m})(\$,v_{m+1})\cdots (\$,v_{n})&
\text{if}\  m<n,
\end{cases}
\end{equation*}
and the mapping $\delta_A^L: A^*\times A^{*}\rightarrow (A(2,
\$))^{*}$ by
\begin{equation*}
(u_{1}\cdots u_{m},v_{1}\cdots v_{n})\mapsto
\begin{cases} (u_{1},v_{1})\cdots (u_{m},v_{n})& \text{if}\  m=n, \\
(u_{1},\$)\cdots (u_{m-n},\$)(u_{m-n+1},v_{1})\cdots (u_{m},v_{n})& \text{if}\  m>n, \\
(\$,v_{1})\cdots (\$,v_{n-m})(u_{1},v_{n-m+1})\cdots (u_{m},v_{n})&
\text{if}\  m<n,
\end{cases}
\end{equation*}
where each $ u_{i},\ v_{j}\in A $.
\end{definition}
Using the mappings $\delta_A^R$ and $\delta_A^L$ defined as above,
we can transform the relation into a language over $A(2,\$)$, which provides a way
to consider automata accepting pairs $(\alpha,\beta)$ of words with $\alpha,\beta\in A^+$ as in the case of automatic groups.
\begin{proposition}\label{22}\rm{(\cite{as})} If $L$ is a regular subset of $A^*$, then
$$
\Delta_L:=\{(\alpha,\alpha)\delta_A^R|\alpha\in L\}=\{(\alpha,\alpha)\delta_A^L|\alpha\in L\}
$$
is regular over $A(2,\$)$.
\end{proposition}

\subsection{Automaticity}

\begin{definition}\label{bia}\rm{(\cite{ABC})} Suppose $S$ is a semigroup with a finite generating set $A$, $L$ is a regular language of $A^+$, and $\phi:A^+\rightarrow S$ is a homomorphism with $\phi(L)=S$. Let
\begin{eqnarray*}
\leftidx{_a^\$}{L}:&=&\{(\alpha,\beta)\delta_A^L|\alpha,\beta\in L,a\alpha=\beta\},\\
\leftidx{^\$}{L_a}:&=&\{(\alpha,\beta)\delta_A^L|\alpha,\beta\in L,\alpha a=\beta\},\\
\leftidx{_a}{L^\$}:&=&\{(\alpha,\beta)\delta_A^R|\alpha,\beta\in L,a\alpha=\beta\},\\
L_a^\$:&=&\{(\alpha,\beta)\delta_A^R|\alpha,\beta\in L,\alpha a=\beta\}.
\end{eqnarray*}
If for any $a\in A\cup\{\varepsilon\},\ _a^\$L\ (\leftidx{^\$}{L}_a,\ \leftidx{_a}{L}^\$,\  L_a^\$,\ resp.$) is regular, then we say semigroup $S$ has a left-left (left-right, right-left, right-right, resp.) automatic structure $(A,L)$. If this is the case, we also  say that $S$ is a {\bf left-left (left-right, right-left, right-right, resp.) automatic semigroup}.
If for any $a\in A\cup\{\varepsilon\}$, $\leftidx{_a}{L^{\$}},\ \leftidx{^\$}{L_a},\ \leftidx{_a}{L^\$}$ and $L_a^\$$ are all regular, then we say semigroup $S$ has a biautomatic structure $(A,L)$ and say that $S$ is a {\bf biautomatic semigroup}.

A  right-right automatic semigroup is also called an {\bf automatic semigroup}.
\end{definition}

\begin{definition}\label{pref}\rm{(\cite{paut})} Suppose $S$ is a semigroup with a finite generating set $A$, $L$ is a regular language of $A^+$, and $\phi:A^+\rightarrow S$ is a homomorphism with $\phi(L)=S$. If $(A, L)$ is an automatic structure for $S$ and
$$
L'_=:=\{(\alpha,\beta)\delta_A^R|\alpha\in L,\ \beta\in Pref(L),\ and\ \alpha=\beta\}
$$
is regular over $A(2,\$)$, then we say that $(A,L)$ is a prefix-automatic structure for $S$ and $S$ is a {\bf prefix-automatic semigroup}.
\end{definition}

For any $\alpha\equiv a_1a_2\cdots a_n\in A^+\ (a_i\in A)$,  $L\subseteq A^+$, and $U\subseteq A^+\times A^+$,
we denote
\begin{eqnarray*}
a^{rev}:&\equiv& a_na_{n-1}\cdots a_1,\\
L^{rev}:&=&\{\alpha^{rev}|\alpha\in L\},\\
U^{rev}:&=&\{(\alpha^{rev}, \beta^{rev})| (\alpha, \beta)\in U\}.
\end{eqnarray*}

\begin{lemma}\label{revaut}\rm{(\cite{ABC})}
\begin{enumerate}
\item[(i)] A semigroup $S$ is  left-left automatic  if and only if $S^{rev}$ is  right-right automatic.
\item[(ii)] A semigroup $S$ is  left-right automatic  if and only if $S^{rev}$ is right-left automatic.
\end{enumerate}
\end{lemma}

\begin{proposition}\label{111}\rm{(\cite{ABC})}
Let $S$ be a semigroup. Then $S$ is (bi-)automatic if and only if $S^1$ is (bi-)automatic.
\end{proposition}

\begin{proposition}\label{fre}\rm{(\cite{as})} Let $S_1$ and $S_2$ be semigroups. Then the free product $S_1*S_2$ of $S_1$ and $S_2$ is right-right (left-left) automatic if and only if both $S_1$ and $S_2$ are right-right (left-left) automatic.
\end{proposition}

\begin{proposition}\label{qua}\rm{(\cite{ascs})} Let $A$ be an alphabet and let $M,N$ be regular languages over $A(2,\$)$. If there exists a constant $C$ such that, for any two words $w_1,w_2\in A^*$, we have
$$(w_1,w_2)\delta_A^R\in M\Rightarrow\ ||w_1|-|w_2||\leq C,$$
then the language
$$
M\odot N:=\{(w_1w_1',w_2w_2')\delta_A^R|(w_1,w_2)\delta_A^R\in M,(w_1',w_2')\delta_A^R\in N\}
$$
is regular.
\end{proposition}

\begin{proposition}\label{sam}\rm{(\cite{ABC})} Let $L\subseteq A^*\times A^*$. Let $k$ be a constant such that
$$||\alpha|-|\beta||\leq k$$
for all $(\alpha,\beta)\in L$. Then $\{(\alpha,\beta)\delta_A^L|(\alpha,\beta)\in L\}$ is regular if and only if $\{(\alpha,\beta)\delta_A^R|(\alpha,\beta)\in L\}$ is regular.
\end{proposition}

\begin{lemma}\label{qup}
 Let $A$ be an alphabet and let $M,N$ be regular languages over $A(2,\$)$. If there exist constants $C$ and $C'$ such that, for any words $w_1,w_2,w_1',w_2'\in A^*$, we have
$$(w_1,w_2)\delta_A^L\in M\delta_A^L \Rightarrow ||w_1|-|w_2||\leq C,$$
$$(w_1',w_2')\delta_A^L\in N\delta_A^L \Rightarrow ||w_1'|-|w_2'||\leq C',$$
then the language $$
M\odot'N:=\{(w_1w_1',w_2w_2')\delta_A^L|(w_1,w_2)\delta_A^L\in M\delta_A^L,(w_1',w_2')\delta_A^L\in N\delta_A^L\}
$$
is regular.
\end{lemma}
\begin{proof}
Since $(w_1,w_2)\delta_A^L\in M\delta_A^L \Rightarrow ||w_1|-|w_2||\leq C$ and   $(w_1',w_2')\delta_A^L\in N\delta_A^L \Rightarrow ||w_1'|-|w_2'||\leq C',$ by Proposition \ref{sam}, we have
$M\delta_A^R$, $N\delta_A^R$ are regular. Then we have $M\odot N$ is regular by Proposition \ref{qua}. For any $(\alpha,\beta)\delta_A^R\in M\odot N$, we have $||\alpha|-|\beta||\leq C+C'$. Hence, by Proposition \ref{sam}, $M\odot' N$ is regular.
\end{proof}

\begin{proposition}\label{cha}\rm{(\cite{ABC})} If $M$ is an automatic monoid and $A$ is any finite  generating set for $M$, then there is a regular language $L$ over $A^+$ such that $(A,L)$ is an automatic structure for $M$.
\end{proposition}

\begin{definition}\rm{(\cite{ABC})} Let $(A,L)$ be an automatic structure for $S$. We say $(A,L)$ is an automatic structure with uniqueness for $S$ if $L$ maps one-to-one to $S$.
\end{definition}

\begin{proposition}\label{uni}\rm{(\cite{ABC})} If $S$ is a semigroup with an automatic structure $(A,L)$, then there exists an automatic structure $(A,K)$ with uniqueness for $S$ with $K\subseteq L$.
\end{proposition}

\subsection{Gr\"{o}bner-Shirshov bases for associative algebras}

To show that a semigroup $S$ is automatic, one needs to find an automatic
structure $(A,L)$, where $A$ is a generating set of $S$ and $L\subset A^+$.
After given a generating set $A$ of $S$,  one usually takes a normal form,
say, $L$, of $S$ to test whether $(A,L)$
is an automatic structure or not. It is known that Gr\"{o}bner-Shirshov
bases theory is a special tool to find formal forms for semigroups.

Let $A$ be a well-ordered set and $F$ be a field. We denote
$F\langle A\rangle$ the free associative algebra over $F$ generated
by $A$.

A well ordering $<$ on $A^*$ is called monomial if for any
$u, v, w\in A^*$, we have
$$
u < v \Rightarrow wu < wv,\ uw < vw.
$$

A classical example of monomial ordering on $A^*$ is the
deg-lex ordering, which first compare two words by degree (length) and
then by comparing them lexicographically.

Let $A^*$ be with a monomial ordering $<$. Then, for any polynomial
$f\in F\langle A\rangle$, $f$ has the leading word $\overline{f}$.
We call $f$  monic if the coefficient of $\overline{f}$ is 1.

Let $f$ and $g$ be two monic polynomials in $F\langle A\rangle$ and
$<$ a monomial ordering on $A^*$. Then, there are two kinds of
compositions:

$ (i)$ If \ $w$ is a word such that $w=\bar{f}b=a\bar{g}$ for some
$a,\ b\in A^*$ with $|\bar{f}|+|\bar{g}|>|w|$, then the polynomial
 $ (f,g)_w=fb-ag$ is called the intersection composition of $f$ and
$g$ with respect to $w$.

$ (ii)$ If  $w=\bar{f}=a\bar{g}b$ for some $a,\ b\in A^*$, then the
polynomial $ (f,g)_w=f - agb$ is called the inclusion
composition of $f$ and $g$ with respect to $w$.

In $(f,g)_w$, $w$ is called ambiguity of the composition.

Let $R\subset F\langle A\rangle$ be a monic subset. Then the
composition $ (f,g)_w$ is called trivial modulo $(R,w)$ if $
(f,g)_w=\sum\alpha_i a_i s_i b_i$, where each $\alpha_i\in F$,
$a_i,\ b_i\in A^{*}, \ s_i\in R$ and $a_i \overline{s_i} b_i<w$.

A monic set $R\subset F\langle A\rangle$ is called a
{\bf Gr\"{o}bner-Shirshov basis} with respect to the monomial
ordering $<$ if any composition of polynomials in $R$ is trivial
modulo $R$ and the corresponding ambiguity.

The following lemma was first proved by Shirshov for free Lie
algebras \cite{Shirshov commutative, Sh62b} (see also  \cite{A,Bo2014}). Bokut \cite{At63}
specialized the approach of Shirshov to associative algebras (see
also Bergman \cite{Al}). For commutative algebras, this lemma is
known as Buchberger's Theorem (see \cite{olshanskii, Bax}).

\begin{lemma}\label{l1}
(Composition-Diamond lemma for associative algebras) \  Let $<$ be a
monomial ordering on $A^*$. Let $R\subset F\langle A\rangle$ be a
nonempty set of monic polynomials and $Id(R)$ the ideal of $F\langle
A\rangle$ generated by $R$. Then   the following statements are
equivalent:
\begin{enumerate}
\item[(1)] $R$ is a Gr\"{o}bner-Shirshov basis in $F\langle
A\rangle$.
\item[(2)] $f\in Id(R)\Rightarrow \bar{f}=a\bar{s}b$
for some $s\in R$ and $a, b\in  A^*$.
\item[(3)] $Irr(R) = \{ u \in A^*\ |\  u \neq a\bar{s}b,\ s\in R,\ a, b \in A^*\}$
is a $F$-basis of the algebra $F\langle A | R \rangle:=F\langle
A\rangle/Id(R)$.
\end{enumerate}
\end{lemma}

If a subset $R$ of $F\langle A \rangle$ is not a
Gr\"{o}bner-Shirshov basis then one can add all nontrivial
compositions of polynomials of $R$ to $R$. Continue this process
repeatedly, we finally obtain a Gr\"{o}bner-Shirshov basis
$R^{comp}$ that contains $R$. Such a process is called
Shirshov's algorithm.

Let $M=sgp\langle A|R\rangle=A^*/\rho(R)$ be a monoid with the identity $\varepsilon$ (the empty word), where $\rho(R)$ is the congruence on $A^*$ generated by $R$. Then $R$ is also a subset
of $F\langle A \rangle$ and we can find a Gr\"{o}bner-Shirshov basis
 $R^{comp}$. We also call $R^{comp}$ a \emph{Gr\"{o}bner-Shirshov basis in $M$}. $Irr(R^{comp})=\{u\in A^*\ |\ u\neq
a\overline{f}b,\ a,\ b \in A^*,\ f\in R^{comp}\}$ is a $F$-basis of
$F\langle A|R\rangle$ which is also a set of  normal forms of
$M$.

Let $sgp^+\langle A|R\rangle=A^+/\rho(R)$ be a semigroup (possibly without identity) generated by $A$ with defining relations $R$, where $\rho(R)$ is the congruence on $A^+$ generated by $R$.
If $R=\{u_{i}=v_{i}|i\in I\}\ (u_{i},v_{i}\in A^+,\ u_{i}>v_{i}$ for any $i\in I$) is a Gr\"{o}bner-Shirshov basis, then
$$
L:=A^{+}-A^{*}\{u_{i}|i\in I\}A^{*}
$$
is  a set of  normal forms of
$M$ which is also called a normal form of $M$. In particular, if $(A,L)$ is an automatic structure for $M$ then $(A,L)$ is also a prefix-automatic structure for $M$.

\section{Main results}

In this section, we denote $sgp\langle A|R\rangle$  the monoid generated by $A$ with defining relations $R$ and $sgp^+\langle A|R\rangle$  the semigroup (possibly without identity) generated by $A$ with defining relations $R$.
Suppose $A$ is a well-ordered set. We use the deg-lex ordering on $A^*$ if we mention Gr\"{o}bner-Shirshov bases. Moreover, if $u=v\in R$, then $u>v$.

 Suppose that $S= sgp^+\langle A|R\rangle$. Then define
$$
S^1=sgp^+\langle A,e|R, ea=ae=a, ee=e, a\in A\rangle.
$$
Clearly, $e$ is  the identity of $S^1$.

\subsection{Some (prefix-)automatic semigroups}
\begin{lemma}\label{preaut}
Let $S=sgp^+\langle A|R\rangle$. Suppose  $L\subseteq A^+$ and $(A, L)$ is a  prefix-automatic structure for $S$. Let $B=A\cup\{e\}$ and $K=L\cup\{e\}$. Then $(B, K)$ is a prefix-automatic structure for $S^1$.
\end{lemma}
\begin{proof}
It is easy to see that $(B, K)$ is an automatic structure for $S^1$, where $e$ maps to the identity of $S^1$. Since $(A, L)$ is a  prefix-automatic structure for $S$, we have
$$
L'_==\{(\alpha,\beta)\delta_A^R|\alpha\in L,\ \beta\in Pref(L),\ and\ \alpha=\beta\}
$$
is regular. Thus,
\begin{eqnarray*}
K'_=&=&\{(\alpha, \beta)\delta_B^R|\alpha\in K,\ \beta\in Pref(K),\ and\ \alpha=\beta\}=\{(e,e)\}\cup L_='
\end{eqnarray*}
is regular. So, $(B, K)$ is a prefix-automatic structure for $S^1$.
\end{proof}

\begin{theorem}\label{3.1.1}
Let $S=sgp^{+}\langle A|u_{i}=v_{i},1\leq i\leq m\rangle$, where $A=\{a_{1},a_{2},\dots,a_{n}\},\ n\in\mathbb{N}$, $u_i, v_i\in A^+$, $i=1,2,\dots,m$ and $\{u_{i}=v_{i}|1\leq i \leq m\}$ is a Gr\"{o}bner-Shirshov basis.
\begin{enumerate}
\item[(i)]\
If $v_{i}(t)\not\equiv u_{j}[t]$ for any $t\geq1$, $i,j\in\{1,2,\dots,m\}$, then S is prefix-automatic.

\item[(ii)]\
If $v_{i}(t)\not\equiv u_{j}[t]$ and $v_{i}[t]\not\equiv u_{j}(t)$ for any $t\geq1$, $i,j\in\{1,2,\dots,m\}$, then $S$ is biautomatic.
\end{enumerate}
\end{theorem}

\begin{proof}
 Since $\{u_{i}=v_{i}|1\leq i \leq m\}$ is a Gr\"{o}bner-Shirshov basis, by Lemma \ref{l1}, we know
$$
L=A^{+}-A^{*}\{u_{i}|1\leq i\leq m\}A^{*}
$$
is a normal form of $S$. Therefore, $L$ maps onto $S$. Obviously, $L$ is a regular language.

If (i) holds, then we show that  $(A,L)$ is a prefix-automatic structure for $S$.

Let $W_j=\{u_i\mid u_i[1]=a_j,1\leq i\leq m\}=\{u_{j_1},u_{j_2},\dots,u_{j_{s_j}}\},\ j=1,2,\dots,n$. By Proposition \ref{22}, $L_{=}^{\$}=\Delta_{L}$ is regular.

Now, by Proposition \ref{22}, $L_{a_{j}}^{\$}=\{(\alpha,\alpha a_{j})\delta_A^{R}|\alpha\in L\}=\Delta_{L}\cdot \{(\$,a_{j})\}$ is regular if $W_j=\emptyset$ and
\begin{eqnarray*}
L_{a_{j}}^{\$}&=&\{(\alpha,\alpha a_{j})\delta_A^{R}|\alpha\in L-\cup_{p=1}^{s_{j}}A^{*}\{u_{j_{p}}(|u_{j_{p}}|-1)\}\}\\
&&\ \ \cup(\cup_{p=1}^{s_{j}}\{(\alpha u_{j_{p}}(|u_{j_{p}}|-1),\ \alpha v_{j_{p}})\delta_A^{R}|\alpha u_{j_{p}}(|u_{j_{p}}|-1)\in L\})\\
&=& \Delta_{L-\cup_{p=1}^{s_j}A^{*}\{u_{j_p}(|u_{j_{p}}|-1)\}}\cdot\{(\$,a_{j})\}\\
&&\ \ \cup((\cup_{p=1}^{s_{j}}\Delta_{L}\cdot\{(u_{j_{p}}(|u_{j_{p}}|-1),v_{j_{p}})\delta_{A}^{R}\})\cap(L\times L)\delta_A^R)
  \end{eqnarray*}
is regular  if $W_j\neq\emptyset$.

Since $L$ is closed under prefix words, $L_='=\Delta_L$ is regular.

  Hence, $S$ is prefix-automatic.

If (ii) holds, then we prove that $(A,L)$ is a biautomatic structure for $S$.

 Let $W'_j=\{u_i\mid
u_i(1)=a_j,\ 1\leq i\leq m\}=\{u_{j_1},u_{j_2},\dots,u_{j_{t_j}}\},\ j=1,2,\dots,n$. Then, by Proposition \ref{22}, $^{\$}L_{=}=\Delta_{L}$ is regular.
Note that
 $$\leftidx{^\$_{a_j}}{L}=\{(\alpha,a_{j}\alpha)\delta_A^{L}|\alpha\in L\}=\{(\$,a_{j})\}\cdot\Delta_{L}$$
 is regular if $W'_j=\emptyset$ and
\begin{eqnarray*}
  \leftidx{_{a_j}^\$}{L}&=&\{(\alpha,a_{j}\alpha)\delta_A^{L}|\alpha\in L-\cup_{p=1}^{t_{j}}\{u_{j_{p}}[|u_{j_{p}}|-1]\}A^{*}\}\\
&&\ \ \cup(\cup_{p=1}^{t_{j}}\{(u_{j_{p}}[|u_{j_{p}}|-1]\alpha, v_{j_{p}}\alpha)\delta_A^{L} |u_{j_{p}}[|u_{j_{p}}|-1]\alpha \in L\})\\
&=& \{(\$,a_{j})\}\cdot\Delta_{L-\cup_{p=1}^{t_j}\{u_{j_p}(|u_{j_{p}}|-1)\}A^{*}}\\
&&\ \ \cup((\cup_{p=1}^{t_{j}}\{(u_{j_{p}}[|u_{j_{p}}|-1],v_{j_{p}})\delta_{A}^{L}\}\cdot\Delta_{L})\cap(L\times L)\delta_A^L)
\end{eqnarray*}
is regular if $W'_j\neq\emptyset$.

Hence, $L_{=}^{\$},\ ^{\$}L_{=},\ L_{a_{j}}^{\$}\ and\ ^{\$}_{a_{j}}L\ (j=1,2,\dots,n)$ are regular.

Let $N=\displaystyle{\mathop{max}_{1\leq i\leq m}}\{|u_{i}|-|v_{i}|\}$. For any $(\alpha,\beta)\delta_A^{R}\in L_{a_{j}}^{\$}\ (j=1,2,\dots,n)$, since $||\alpha|-|\beta||\leq N+1$, by Proposition \ref{sam}, we have $^{\$}L_{a_{j}}\ (j=1,2,\dots,n)$ is regular. Similarly, $ _{a_{j}}L^{\$}\ (j=1,2,\dots,n)$ is regular.

Hence, $(A,L)$ is a biautomatic structure for $S$.
\end{proof}

  \ \

Now, we consider some one-relator semigroups.

\begin{theorem}\label{3.1.3}
 Let $S=sgp^{+}\langle a_1,a_2,\dots,a_n|a_{i_1}a_{i_2}\cdots a_{i_k}=a_{j_1}a_{j_2}\cdots a_{j_k}\rangle$, where $n\in\mathbb{N}$, $k\geq2$, $a_{i_l}, a_{j_l}\in \{a_1, a_2, \dots, a_n\}$, $l=1,2,\dots, k$ and $\{a_{i_1}a_{i_2}\cdots a_{i_k}=a_{j_1}a_{j_2}\cdots a_{j_k}\}$ is a Gr\"{o}bner-Shirshov basis. If $S$ satisfies one of the following conditions, then $S$ is biautomatic and prefix-automatic.
\begin{enumerate}
\item[(1)] For any $t\geq1,\ (a_{j_1}a_{j_2}\cdots a_{j_k})(t)\not\equiv (a_{i_1}a_{i_2}\cdots a_{i_k})[t]$\ and\ $(a_{j_1}a_{j_2}\cdots a_{j_k})[t]\not\equiv (a_{i_1}a_{i_2}\cdots a_{i_k})(t)$;

\item[(2)] $|con(a_{i_1}a_{i_2}\cdots a_{i_k})|=k$;

\item[(3)]\ $k=2$;

\item[(4)]\ $con(a_{i_1}a_{i_2}\cdots a_{i_k})\nsubseteq con(a_{j_1}a_{j_2}\cdots a_{j_k})$;

\item[(5)]\ $a_{i_1}a_{i_2}\cdots a_{i_k} \not\equiv ww'w$ and $a_{j_1}a_{j_2}\cdots a_{j_k} \not\equiv ss's$ for any $w,s\in \{a_1,a_2,\dots,a_n\}^{+},w',$ $s'\in \{a_1,a_2,\dots,a_n\}^{*}$.
\end{enumerate}
\end{theorem}

\begin{proof}
Let $A=\{a_1,a_2,\dots,a_n\}$ and $L=A^{+}-A^{*}\{a_{i_1}a_{i_2}\cdots a_{i_k}\}A^{*}$. Since $\{a_{i_1}a_{i_2}\cdots a_{i_k}=a_{j_1}a_{j_2}\cdots a_{j_k}\}$ is a Gr\"{o}bner-Shirshov basis, $L$ is a normal form of $S$.

By Proposition \ref{22}, $L_{=}^{\$}=\Delta_{L}=^{\$}L_{=}$ and
\begin{eqnarray*}
L_{a_j}^{\$}&=&\{(\alpha,\alpha a_j)\delta_A^{R}|\alpha\in L\}=\Delta_{L}\cdot\{(\$,a_j)\},\ j\neq  i_k ,\\
^{\$}_{a_j}L&=&\{(\alpha, a_j\alpha)\delta_A^{L}|\alpha\in L\}=\{(\$,a_j)\}\cdot\Delta_{L},\  j\neq  i_1
\end{eqnarray*}
are regular. Now,
\begin{eqnarray*}
L_{a_{i_k}}^{\$}&=&\{(\alpha a_{i_1}a_{i_2}\cdots a_{i_{k-1}},\beta)\delta_A^{R}|\alpha a_{i_1}a_{i_2}\cdots a_{i_{k-1}}, \beta\in L,\alpha a_{i_1}a_{i_2}\cdots a_{i_k}=\beta\}\\
&&\ \cup\{(\alpha,\alpha a_{i_k})\delta_A^{R}|\alpha\in L-A^{*}\{a_{i_1}a_{i_2}\cdots a_{i_{k-1}}\}\},\\
^{\$}_{a_{i_1}}L&=&\{( a_{i_2}a_{i_3}\cdots a_{i_{k}}\alpha,\beta)\delta_A^{L}| a_{i_2}a_{i_3}\cdots a_{i_{k}}\alpha,\beta\in L,a_{i_1}a_{i_2}\cdots a_{i_k}\alpha =\beta\}\\
&&\ \cup\{(\alpha, a_{i_1}\alpha)\delta_A^{L}|\alpha\in L-\{a_{i_2}a_{i_3}\cdots a_{i_{k}}\}A^{*}\}.
\end{eqnarray*}
Denote
 \begin{eqnarray*}
W_R&=&\{(\alpha a_{i_1}a_{i_2}\cdots a_{i_{k-1}},\beta)\delta_A^{R}|\alpha a_{i_1}a_{i_2}\cdots a_{i_{k-1}},\beta\in L,\alpha a_{i_1}a_{i_2}\cdots a_{i_k}=\beta\},\\
W_L&=&\{( a_{i_2}a_{i_3}\cdots a_{i_{k}}\alpha,\beta)\delta_A^{L}| a_{i_2}a_{i_3}\cdots a_{i_{k}}\alpha,\beta\in L,a_{i_1}a_{i_2}\cdots a_{i_k}\alpha =\beta\}.
\end{eqnarray*}
 If $W_R$ and $W_L$ are regular, then $L_{a_{i_k}}^{\$}$ and $^{\$}_{a_{i_1}}L$ are regular. Since $a_{i_1}a_{i_2}\cdots a_{i_k}=a_{j_1}a_{j_2}\cdots a_{j_k}$ is homogeneous, for any $(\alpha,\beta)\delta_A^{R}\in L_{a_{i}}^{\$}$ (or $(\alpha,\beta)\delta_A^{L}\in$ $^{\$}_{a_i}L),\ i=1,2,\dots,n$, we have $||\alpha|-|\beta||\leq 1$. Then by Proposition \ref{sam}, we have $^{\$}L_{a_i}$ and $_{a_i}L^{\$}$ are regular. Therefore, $(A,L)$ is a biautomatic structure for $S$. That is to say, to prove $S$ is biautomatic, it is sufficient to prove that $W_R$ and $W_L$ are regular.

Now, we prove that $W_R$ and $W_L$ are regular if $S$ satisfies one of conditions $(1)-(5)$.

Denote $u\equiv a_{i_1}a_{i_2}\cdots a_{i_k}$, $v\equiv a_{j_1}a_{j_2}\cdots a_{j_k}$.

Case $1.$ If $S$ satisfies condition $(1)$, then by Theorem \ref{3.1.1}, $S$ is biautomatic.

Case $2.$ If there exist $l,l'\in \mathbb{N},\ t_1,t_2,\dots,t_l\in\{1,2,\dots,k-1\}$ and $s_1,s_2,\dots,s_{_{l'}}\in\{1,2,\dots,k-1\}$ such that $v(t_i)\equiv u[t_i],\ v[s_j]\equiv u(s_j)$ for each $i\in\{1,2,\dots,l\}$ and each $j\in\{1,2,\dots,l'\}$, then let
\begin{eqnarray*}
 \alpha &\equiv& \alpha'( u(k-t_1)^{p_{11}}u(k-t_2)^{p_{12}}\cdots u(k-t_l)^{p_{1l}} )( u(k-t_1)^{p_{21}}u(k-t_2)^{p_{22}}\cdots\\
         &&u(k-t_l)^{p_{2l}})\cdots( u(k-t_1)^{p_{c1}}u(k-t_2)^{p_{c2}}\cdots u(k-t_l)^{p_{cl}} )u(k-1),\\
\beta &\equiv& \alpha'v( v[k-t_1]^{p_{11}}v[k-t_2]^{p_{12}}\cdots v[k-t_l]^{p_{1l}})( v[k-t_1]^{p_{21}}v[k-t_2]^{p_{22}}\cdots \\
    &&v[k-t_l]^{p_{2l}})\cdots( v[k-t_1]^{p_{c1}}v[k-t_2]^{p_{c2}}\cdots v[k-t_{l}]^{p_{cl}}),\\
\sigma &\equiv& (u[k-1]( u[k-s_1]^{p_{11}'}u[k-s_2]^{p_{12}'}\cdots u[k-s_{_{l'}}]^{p_{1l'}'} )( u[k-s_1]^{p_{21}'}u[k-s_2]^{p_{22}'}\\
    &&\cdots u[k-s_{_{l'}}]^{p_{2l'}'})\cdots( u[k-s_1]^{p_{c'1}'}u[k-s_2]^{p_{c'2}'}\cdots u[k-s_{_{l'}}]^{p_{c'l'}'} )\alpha'',\\
\gamma &\equiv& (v(k-s_1)^{p_{11}'}v(k-s_2)^{p_{12}'}\cdots v(k-s_{_{l'}})^{p_{1l'}'})( v(k-s_1)^{p_{21}'}v(k-s_2)^{p_{22}'}\cdots \\
    &&v(k-s_{_{l'}})^{p_{2l'}'})\cdots( v(k-s_1)^{p_{c'1}'}v(k-s_2)^{p_{c'2}'}\cdots v(k-s_{_{l'}})^{p_{c'l'}'})v\alpha''),
\end{eqnarray*}
where $p_{11}, p_{12}, \dots,p_{1l}, \dots, p_{c1}, p_{c2}, \dots, p_{cl}\geq 0,\ c\in \mathbb{N}$, $p_{11}', p_{12}', \dots, p_{1l'}', \dots, p_{c'1}', p_{c'2}', \dots, $
$p_{c'l'}'\geq 0,\ c'\in \mathbb{N}$, and $\alpha'\in L-A^{*}\{u(k-t_1),u(k-t_2),\dots, u(k-t_l)\}$, $\alpha^{''}\in L-\{u[k-s_1],u[k-s_2],\dots, u[k-s_{_{l'}}]\}A^{*}$. Thus, $\alpha\cdot a_{i_k}=\beta$ and $a_{i_1}\cdot\sigma=\gamma$. By noting that
\begin{eqnarray*}
\widetilde{W}_{R}&=&\{(\alpha,\beta)\delta_A^{R} | \alpha\in L,\ \alpha'\in L-A^{*}\{u(k-t_1),u(k-t_2),\dots, u(k-t_l)\}\\
&&\ p_{11},\dots,p_{1l}, \dots,p_{c1},\dots,p_{cl}\geq0\}\\
&=&(\Delta_{ L-A^{*}\{u(k-t_1),u(k-t_2),\dots, u(k-t_l)\}  } \{(\varepsilon, v)\delta_{A}^{R} \} )\odot
   \{ \{(u(k-t_1),v[k-t_1])\delta_{A}^{R}\}^{*}\\
   &&\{(u(k-t_2),v[k-t_2])\delta_{A}^{R}\}^{*}\cdots\{(u(k-t_l),v[k-t_l])\delta_{A}^{R}\}^{*} \}^{*} \{(u(k-1),\varepsilon)\delta_{A}^{R}\},\\
\widetilde{W}_{L}&=&\{(\sigma,\gamma)\delta_A^{L} |\sigma\in L,\ \alpha^{''}\in L-\{u[k-s_1],u[k-s_2],\dots, u[k-s_{_{l'}}]\}A^{*}\\
&&\ \ p_{11}',\dots,p_{1l'}',\dots,p_{c'1}',\dots,p_{c'l'}'\geq0\}\\
&=&\{(u[k-1],\varepsilon)\delta_{A}^{L}\}\{\{(u[k-s_1],v(k-s_1))\delta_{A}^{L}\}^{*}\{(u[k-s_2],v(k-s_2))\delta_{A}^{L}\}^{*}\cdots \\
&&\ \ \{(u[k-s_{_{l'}}],v(k-s_{_{l'}}))\delta_{A}^{L}\}^{*} \}^{*}\odot'\{(\varepsilon,v)\delta_{A}^{L}\}\Delta_{ L-\{u[k-s_1],u[k-s_2],\dots, u[k-s_{_{l'}}]\}A^{*}  }
 \end{eqnarray*}
are regular by   Proposition \ref{qua} and Lemma \ref{qup}, we just need to prove that $\beta,\gamma\in L$.

$(i)$ Suppose $S$ satisfies condition $(2)\ |con(u)|=|u|$.

If $\beta\not\in L$, then $u$ is a subword of $\beta$  and $u$ must be of the form
$$v[k-t_{i_1}][s]v[k-t_{i_2}]\cdots v[k-t_{i_h}](s'),$$
where $t_{i_j}\in \{t_1,t_2,\dots,t_l\}$, $h\geq2$ and $s+(k-t_{i_2})+(k-t_{i_3})+\cdots+(k-t_{i_{h-1}})+s'=k$.
If $h=2$, then $u\equiv v[t]v[t'](s)$, where $0<t,t',s<k$. Obviously, $t+s=k$. Since $|con(u)|=|u|$, we have $t'\geq t+s=k$ (Otherwise $con(v[t])\cap con(v[t'](s))\neq\emptyset$ which contradicts $|con(u)|=|u|$) which contradicts $t'<k$.
If $h>3$, then $u\equiv v[t_1']v[t_2']\cdots v[t_h'](s)$. Obviously, we have $con(v[t_1'])\cap con(v[t_2'])\neq\varnothing$, a contradiction.
Hence $\beta\in L$.

If $\gamma\not\in L$, then $u$ is a subword of $\gamma$ and $u$ must be of the form
$$
v(k-s_{i_1})[t]v(k-s_{i_2})\cdots v(k-s_{i_{h'}})(t'),
$$
where $s_{i_j}\in \{s_1,s_2,\dots,s_{_{l'}}\}$, $h'\geq2$ and $t+(k-s_{i_2})+(k-s_{i_3})+\cdots+(k-s_{i_{h'-1}})+t'=k$.
If $h'=2$, then $u\equiv v(s)[t]v(s')$, where $0<s,t,s'<k$. Obviously, $t+s'=k$. Since $|con(u)|=|u|$, we have $s\geq s'+t=k$ (otherwise $con(v(s'))\cap con(v(s)[t])\neq\emptyset$, a contradiction) which contradicts $s<k$.
If $h'>3$, then $u\equiv v(s_1')[s]v(s_2')\cdots v(s_h')$. Obviously, we have $con(v(s_2'))\cap con(v(s_3'))\neq\varnothing$, a contradiction.
Hence $\gamma\in L$.

$(ii)$ Suppose $S$ satisfies condition $(3)\ |u|=2$.

If $\beta\not\in L$, then $u$ is a subword of $\beta$ and $u$ must be of the form $v[1]v[1]$. Suppose $v\equiv ca$. Then $u\equiv aa$. Since $aaa=caa=cca$ and $aaa=aca$, we have $cca=aca$. Since $\{u=v\}$ is a Gr\"{o}bner-Shirshov basis, we have $cca\equiv aca$. Then $a\equiv c$. This contradicts $u\not\equiv v$. Hence $\beta\in L$.

If $\gamma\not\in L$, then $u$ is a subword of $\gamma$ and $u$ must be of the form $v(1)v(1)$. Suppose $v\equiv ac$. Then $u\equiv aa$. Since $aaa=aca$ and $aaa=aac=acc$, we have $acc\equiv aca$ and so $a\equiv c$, that is, $u\equiv v$, a contradiction. Hence $\gamma\in L$.

$(iii)$ Suppose $S$ satisfies condition $(4)\ con(u)\nsubseteq con(v)$.

If $\beta\not\in L$  ($\gamma\not\in L$, resp.), then $u$ is a subword of $\beta$ ($\gamma$, resp.)  and $u$ must be contained in some subword of $\beta$ of the form $v[k-t_{i_1}]v[k-t_{i_2}]\cdots v[k-t_{i_h}]$ (in some subword of $\gamma$ of the form $v(k-s_{i_1})v(k-s_{i_2})\cdots v(k-s_{i_{h'}})$, resp.). This  contradicts $con(u)\nsubseteq con(v)$. Hence  $\beta,\gamma\in L$.

$(iv)$ Suppose S satisfies condition $(5)\ a_{i_1}a_{i_2}\cdots a_{i_k} \not\equiv ww'w,\  a_{j_1}a_{j_2}\cdots a_{j_k} \not\equiv ss's$ for any $w,s\in \{a_1,a_2,\dots,a_n\}^{+},w',s'\in \{a_1,a_2,\dots,a_n\}^{*}$.

 \begin{center}

\includegraphics[width=0.4\textwidth,angle=-0]{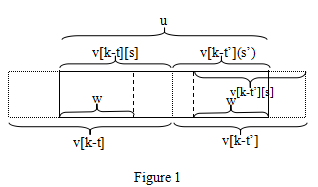}\ \ \ \ \ \ \ \
\includegraphics[width=0.4\textwidth,angle=-0]{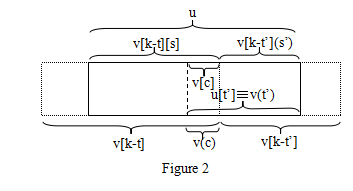}

\includegraphics[width=0.4\textwidth,angle=-0]{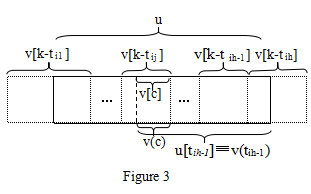}\ \ \ \ \ \ \ \
\includegraphics[width=0.4\textwidth,angle=-0]{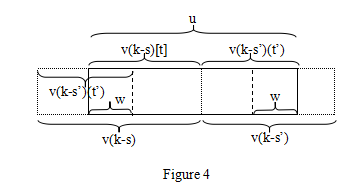}

\includegraphics[width=0.4\textwidth,angle=-0]{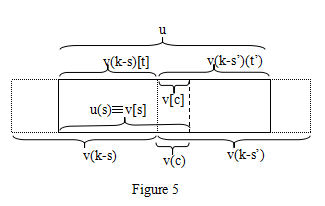}\ \ \ \ \ \ \ \
\includegraphics[width=0.4\textwidth,angle=-0]{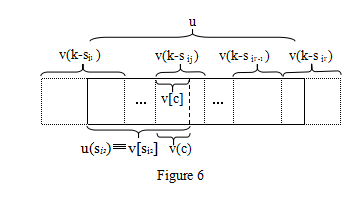}

 \end{center}

If $\beta\not\in L$, then $u\equiv v[k-t_{i_1}][s]v[k-t_{i_2}]\cdots v[k-t_{i_h}](s')$, where $t_{i_j}\in \{t_1,t_2,\dots,t_l\}$, $h\geq2$, $1\leq s\leq k-t_{i_1}$ and $1\leq s'\leq k-t_{i_h}$.

If $h=2$, then $u\equiv v[k-t][s]v[k-t'](s')$ and $s+s'=k>k-t'$. Hence $s>k-t'-s'$.

Case $1.$ If $k-t'\geq s$, then $k-t'\geq s>k-t'-s'$. Hence there exists a prefix $w$ of $v[s]$ such that $u\equiv ww'w$ for some $w'\in A^{*}$, see Figure $1$. This contradicts condition $(5)$.

Case $2.$ If $k-t'< s$, then $k-s=s'< t'$. Since $u[t']\equiv v(t')$, there exists a suffix $v[c]$ of $v[k-t]$ such that $v[c]\equiv v(c)$. Hence, there exists a prefix $w$ of $v(c)$ such that $v\equiv ww'w$ (if $c\leq \frac{k}{2}, w\equiv v(c)$; if $c\geq \frac{k}{2},w\equiv v(2c-k)$) for some $w'\in A^{*}$, see Figure $2$. This is a contradiction.

If $h\geq3$, then $u\equiv v[k-t_{i_1}][s]v[k-t_{i_2}]\cdots v[k-t_{i_h}](s')$. Since $v[k-t_{i_{h-1}}]$ is completely contained in $u$, we have $t_{i_{h-1}}=k-(k-t_{i_{h-1}})>s'$. Since $u[t_{i_{h-1}}]\equiv v(t_{i_{h-1}})$, there exists a suffix $v[c]$ of $v[t_{i_j}]$ for some $t_{i_j}\in \{t_{i_1},t_{i_2},\dots,t_{i_{h-1}}\}$ such that $v[c]\equiv v(c)$, see Figure $3$. That is $v\equiv ww'w$ (if $c\leq \frac{k}{2}, w\equiv v(c)$; if $c\geq \frac{k}{2},w\equiv v(2c-k)$) for some $w'\in A^{*}$, a contradiction.

Hence $\beta\in L$.

If $\gamma\not\in L$, then $u\equiv v(k-s_{i_1})[t]v(k-s_{i_2})\cdots v(k-s_{i_{h'}})(t')$, where $s_{i_j}\in \{s_1,s_2,\dots,s_{_{l'}}\}$, $h'\geq2$, $1\leq t\leq k-s_{i_1}$ and $1\leq t'\leq k-s_{i_{h'}}$.

If $h'=2$, then $\gamma\equiv v(k-s)[t]v(k-s')(t')$ and $t+t'=k>k-s$. Hence $t'>k-s-t$.

Case $1.$ If $k-s\geq t'$, then $k-s\geq t'>k-s-t$. Hence there exists a suffix $w$ of $v(t')$ such that $u\equiv ww'w$ for some $w'\in A^{*}$, see Figure $4$.

Case $2.$ If $k-s< t'$, then $k-t'=t< s$. Since $u[s]\equiv v(s)$,  there exists a prefix $v(c)$ of $v[k-s']$ such that $v[c]\equiv v(c)$. Hence, there exists a prefix $w$ of $v(c)$ such that $v\equiv ww'w$ (if $c\leq \frac{k}{2}, w\equiv v(c)$; if $c\geq \frac{k}{2},w\equiv v(2c-k)$) for some $w'\in A^{*}$, see Figure $5$.

If $h'\geq3$, then $u\equiv v(k-s_{i_1})[t]v(k-s_{i_2})\cdots v(k-s_{i_{h'}})(t')$. Since $v(k-s_{i_{2}})$ is completely contained in $u$, we have $s_{i_{2}}=k-(k-s_{i_{2}})>t$. Since $u(s_{i_{2}})\equiv v[s_{i_{2}}]$, there exists a prefix $v(c)$ of $v(s_{i_j})$ for some $s_{i_j}\in \{s_{i_1},s_{i_2},\dots,s_{i_{h'-1}}\}$ such that $v(c)\equiv v[c]$, see Figure $6$. That is $v\equiv ww'w$ (If $c\leq \frac{k}{2}, w\equiv v(c)$; if $c\geq \frac{k}{2},w\equiv v(2c-k)$) for some $w'\in A^{*}$.

Both cases contradict condition (5). Hence $\gamma\in L$.

Therefore, $\beta,\gamma\in L$ if $S$ satisfies one of conditions $(2)-(5)$. This shows that $\widetilde{W}_R=W_R$ and $\widetilde{W}_L=W_L$. Hence $W_R$ and $W_L$ are regular. Thus, $S$ is biautomatic.

Since $L$ is closed under prefix words,  $S$ is prefix-automatic.
\end{proof}

\begin{theorem}\label{3.1.4}
Suppose that $S=sgp\langle a_1,a_2,\dots,a_n|a_{i_1}a_{i_2}\cdots a_{i_k}=\varepsilon\rangle,\ n\in\mathbb{N},\ k\geq2,\ a_{i_j}\in \{a_1, a_2, \dots, a_n\},\ j=1,2,\dots,k$ and $\{a_{i_1}a_{i_2}\cdots a_{i_k}=\varepsilon\}$ is a Gr\"{o}bner-Shirshov basis. Then $S$ is biautomatic.
\end{theorem}

\begin{proof}
Since $S=sgp\langle a_1,a_2,\dots,a_n|a_{i_1}a_{i_2}\cdots a_{i_k}=\varepsilon\rangle\cong\widetilde{S}=sgp^{+}\langle e, a_1,a_2,\dots,a_n|a_{i_1}a_{i_2}$
$\cdots a_{i_k}=e,ee=e,a_je=a_j=ea_j,j=1,2,\dots,n\rangle$, we just need to prove that $\widetilde{S}$ is biautomatic.

Let $A=\{e,a_1,a_2,\dots,a_n\}$, $L=\{e\}\cup(A^{+}-A^{*}\{a_{i_1}a_{i_2}\cdots a_{i_k},e\}A^{*})$. It is easy to see that $\{a_{i_1}a_{i_2}\cdots a_{i_k}
=e,ee=e,a_je=a_j=ea_j,j=1,2,\dots,n\}$ is also a Gr\"{o}bner-Shirshov basis in $\widetilde{S}$. Thus, $L$ is a normal form of $\widetilde{S}$. Obviously, $L$ is regular and $L$ maps onto $\widetilde{S}$. We  prove that $(A,L)$ is a biautomatic structure for $\widetilde{S}$.

By Proposition \ref{22}, $L_{=}^{\$}=\Delta_{L}=^{\$}L_{=}$ is regular. Note that
\begin{eqnarray*}
L_{a_{i}}^{\$}&=&\{(e,a_i)\}\cup\{(\alpha,\alpha a_{i})\delta_A^{R}\ |\ \alpha\in L-\{e\}\}=\{(e,a_i)\}\cup(\Delta_{L-\{e\}}\cdot \{(\$,a_{i})\}),\\
_{a_{j}}^{\$}L&=&\{(e,a_j)\}\cup\{(\alpha,a_{j}\alpha)\delta_A^{L}\ |\ \alpha\in L-\{e\}\}=\{(e,a_j)\}\cup(\{(\$,a_{j})\}\cdot\Delta_{L-\{e\}})
\end{eqnarray*}
are regular for any $a_{i}\neq a_{i_k}$ and $a_{j}\neq a_{i_1}$. Also,
\begin{eqnarray*}
L_{a_{i_k}}^{\$}&=&\{(e,a_{i_k})\}\cup\{(\alpha,\alpha a_{i_k})\delta_A^{R}\ |\ \alpha\in L-(\{e\}\cup A^{*}\{a_{i_1}a_{i_2}\cdots a_{i_{k-1}}\})\}\\
&& \cup\{(\alpha a_{i_1}a_{i_2}\cdots a_{i_{k-1}},\alpha)\delta_A^{R}\ |\ \alpha a_{i_1}a_{i_2}\cdots a_{i_{k-1}}\in L\}\\
&=&\{(e,a_{i_k})\}\cup(\Delta_{L-(\{e\}\cup A^{*}\{a_{i_1}a_{i_2}\cdots a_{i_{k-1}}\})}\cdot \{(\$,a_{i_k})\})\\
&& \cup((\Delta_L\cdot\{(a_{i_1}a_{i_2}\cdots a_{i_{k-1}},\varepsilon)\delta_A^{R}\})\cap (L\times L)\delta_A^{R}),\\
 _{a_{i_1}}^{\$}L&=&\{(e,a_{i_1})\}\cup\{(\alpha,a_{i_1}\alpha)\delta_A^{L}\ |\ \alpha\in L-(\{e\}\cup\{a_{i_2}a_{i_3}\cdots a_{i_k}\}A^{*})\}\\
 && \cup\{(a_{i_2}a_{i_3}\cdots a_{i_k}\alpha,\alpha)\delta_A^{L}\ |\ a_{i_2}a_{i_3}\cdots a_{i_k}\alpha\in L\}\\
 &=&\{(e,a_{i_1})\}\cup(\{(\$,a_{i_1})\}\cdot\Delta_{L-\{e\}\cup\{a_{i_2}a_{i_3}\cdots a_{i_k}\}A^{*}\}})\\
 &&\cup((\{(a_{i_2}a_{i_3}\cdots a_{i_{k}},\varepsilon)\delta_A^{L}\}\cdot\Delta_L)\cap (L\times L)\delta_A^{L})
\end{eqnarray*}
are regular.

For any $(\alpha,\beta)\delta_A^{R}\in L_{a_{j}}^{\$}$ $(or\ (\alpha,\beta)\delta_A^{L} \in$ $^{\$}_{a_{j}}L)$, since $||\alpha|-|\beta||\leq k$, by Proposition \ref{sam}, we have $^{\$}L_{a_{j}}\ (or\ _{a_{j}}L^{\$})$ is regular, $j=1,2,\dots,n$.

Hence $(A,L)$ is a biautomatic structure for $\widetilde{S}$.
\end{proof}

\begin{theorem}\label{3.1.5}
Suppose that $S=sgp^{+}\langle a_1,a_2,\dots,a_n|a_{i_1}a_{i_2}\cdots a_{i_k}=x\rangle,\ n\in \mathbb{N},\ k\geq2,\ x, a_{i_j}\in \{a_1, a_2, \dots, a_n\}$, $j=1,2,\dots,k$ and $\{a_{i1}a_{i2}\cdots a_{ik}=x\}$ is a Gr\"{o}bner-Shirshov basis. Then $S$ is prefix-automatic.
\end{theorem}
\begin{proof}
Case 1. If $x\not\equiv a_{ik}$, then by Theorem \ref{3.1.1}, S is prefix-automatic.

Case 2. Let $x\equiv a_{ik}$ and $A=\{a_1,a_2,\dots,a_n\}$. Then $L=A^+-A^*\{a_{i_1}a_{i_2}\cdots a_{i_k}\}A^*$ is a normal form of $S$ and $L_=^\$=\Delta_L$ is regular by Proposition \ref{22}. By noting that
$$L_{a_j}^\$=\{(\alpha,\alpha a_j)\delta_A^R|\alpha\in L\}=\Delta_L\cdot\{(\$,a_j)\}$$
is also regular for any $a_j\neq a_{i_k}$ and
\begin{eqnarray*}
L_{a_{i_k}}^\$ &=&\{(\alpha(a_{i_1}a_{i_2}\cdots a_{i_{k-1}})^i,\alpha a_{i_k})\delta_A^R|\alpha\in L-A^*\{a_{i_1}a_{i_2}\cdots a_{i_{k-1}}\},\\
&&\ \ \alpha(a_{i_1}a_{i_2}\cdots a_{i_{k-1}})^i\in L,i\geq 1\}\\
&&\cup\{(\alpha,\alpha a_{i_k})\delta_A^R|\alpha\in L-A^*\{a_{i_1}a_{i_2}\cdots a_{i_{k-1}}\}\}\\
&=&(((\Delta_{L-A^*\{a_{i_1}a_{i_2}\cdots a_{i_{k-1}}\}}\{\$,a_{i_k}\})\odot\{(a_{i_1}a_{i_2}\cdots a_{i_{k-1}},\ \varepsilon)\delta_A^R\}^+)\cap(L\times L)\delta_A^R)\\
&& \cup\Delta_{L-A^*\{a_{i_1}a_{i_2}\cdots a_{i_{k-1}}\}}\{(\$,a_{i_k})\}
\end{eqnarray*}
is regular by Propositions \ref{22} and  \ref{qua}, $S$  is automatic. Since $L$ is closed under prefix words, $S$ is prefix-automatic.
\end{proof}

Noting that the semigroup $S$ in Theorem \ref{3.1.5} may not be biautomatic, the following theorem is an example.

\begin{theorem}\label{3.1.8}
Let $S=sgp^+\langle a,b|a^kb=b\rangle$, where $k\geq1$. Then $S$ is prefix-automatic but not biautomatic.
\end{theorem}

\begin{proof} Clearly, $\{a^kb=b\}$ is a Gr\"obner-Shirshov basis.
By Theorem \ref{3.1.5}, $S$ is prefix-automatic. We now show that there does not exist  biautomatic structure for $S$.

Suppose $S$ is biautomatic. Then $S^1$ is biautomatic by Proposition \ref{111}. For $B=\{e,a,b\}$, there exists
$K\subseteq B^+$ such that $(B,K)$ is a biautomatic structure for $S^1$ with uniqueness by Propositions \ref{cha} and  \ref{uni}. So $\leftidx{^\$}K_b$ is regular. For $i,j\in\mathbb{N}$, let $\alpha,\beta\in K$ with $\alpha=b^j(a^k)^i,\beta=b^{j+1}$. Then we have $(\alpha,\beta)\in \leftidx{^{\$}}K_b $ and
\begin{eqnarray*}
\alpha&\equiv&\gamma_1b\gamma_2b\cdots\gamma_jb\gamma_{j+1}(a\tau_{11}a\tau_{12}\cdots a\tau_{1k})\cdots(a\tau_{i1}a\tau_{i2}\cdots a\tau_{ik}),\\
\beta&\equiv&\eta_1b\eta_2b\cdots\eta_{j+1}b\eta_{j+2},
\end{eqnarray*}
where $\gamma_{j+1},\tau_{11},\dots,\tau_{1k},\dots,\tau_{i1},\dots,\tau_{ik},\eta_{j+2}\in\{e\}^*,
\gamma_1,\dots,\gamma_j,\eta_1,\dots,\eta_{j+1}\in\{e,a\}^*$. Denote
$\alpha_1\equiv\gamma_{j+1}(a\tau_{11}a\tau_{12}\cdots a\tau_{1k})\cdots(a\tau_{i1}a\tau_{i2}\cdots a\tau_{ik})$. Then
 $ik\leq|\alpha_1|\leq N(ik+1)+ik$ and
$j+1\leq|\beta|<N(j+2)+j+1$,
where $N=S(M(K))$.

Let $ik>N(j+2)+j+1$ and $j>|S(M(^\$K_b))|$. Then there will be a loop $(u_1,u_2)\delta_B^L$ in $(\alpha_1,\beta)\delta_B^L$. Assume
$\alpha\equiv w_1u_1w_2,\ \beta\equiv w_1'u_2w_2'$. Since $u_1$ is a subword of $\alpha_1$, $b\not\in con(u_1)$. If $b\in con(u_2)$, we have $occ(b,w_1u_1^iw_2b)\neq occ(b,w_1'u_2^iw_2')$, so $(w_1u_1^iw_2,w_1'u_2^iw_2')\delta_B^R\not\in\newcommand{\leftexp}[2]{{\vphantom{#2}}^{#1}{#2}}\leftexp{\$}{K}_b$ for
$i>1$, a contradiction. Hence $b\not\in con(u_2)$ and so $con(u_2)\subseteq\{e,a\}$. Therefore $\beta\equiv w_1'u_2w_2'=w_1'u_2^{k+1}w_2'\in K$  which contradicts the uniqueness of $K$. Thus $S$ is not biautomatic.
\end{proof}

\begin{theorem}\label{3.1.6}
Let $S=sgp^+\langle a_1,a_2,\dots,a_n|a_{i_1}a_{i_2}\cdots a_{i_k}=xy\rangle$, $n\in\mathbb{N},\ k\geq2$, where $x, y, a_{i_1}, \dots, a_{i_k}\in \{a_1, a_2, \dots, a_n\}$ and $\{a_{i_1}a_{i_2}\cdots a_{i_k}=xy\}$ is a Gr\"{o}bner-Shirshov basis. If $a_{i_{k-1}}a_{i_k}\not\equiv xy$ and $a_{i_1}a_{i_2}\cdots a_{i_k}\not\equiv y^{k-1}x$, then $S$ is automatic and $S^1$ is prefix-automatic.
\end{theorem}

\begin{proof}
If $x\not\equiv a_{i_k}$,   then $S$ is prefix-automatic by Theorem \ref{3.1.1}. We now suppose $x\equiv a_{i_k}$.

Let $A=\{a_1,a_2,\dots,a_n\}$ and
$L=A^+-A^*\{a_{i_1}a_{i_2}\cdots a_{i_{k-1}}x\}A^*$. Then $L$ is a normal form of $S$ and
$L$ is clearly regular.

Case 1. If $x\equiv y$, then $S=sgp^+\langle a_1,a_2,\dots,a_n|a_{i_1}a_{i_2}\cdots a_{i_{k-1}}x=x^2\rangle$  since $a_{i_{k-1}}x\not\equiv xx$ and $a_{i_{k-1}}\not\equiv x$. Clearly, $L_=^\$=\Delta_L$ and $L_{a_j}^{\$}=\Delta_L\cdot\{(\$,a_j)\}$ are regular if $a_j\not\equiv x$.

Case 1-1. If $a_{i_1}a_{i_2}\cdots a_{i_{k-1}}x$ is a subword of $a_{i_1}a_{i_2}\cdots a_{i_{k-1}}a_{i_1}a_{i_2}\cdots a_{i_{k-1}}$, then
\begin{eqnarray*}
L_x^\$&=&\{(\alpha,\alpha x)\delta_A^R|\alpha\in L-A^*\{a_{i_1}a_{i_2}\cdots a_{i_{k-1}}\}\}\\
&&\cup\{(\alpha a_{i_1}a_{i_2}\cdots a_{i_{k-1}},\alpha xx)\delta_A^R|\alpha a_{i_1}a_{i_2}\cdots a_{i_{k-1}}\in L,\alpha\in L-A^*\{a_{i_1}a_{i_2}\cdots a_{i_{k-1}}\}\}
\end{eqnarray*}
is regular by   Proposition \ref{22}. Hence  $S$ is prefix-automatic.

Case 1-2. If $ a_{i_1}a_{i_2}\cdots a_{i_{k-1}}x$ is not a subword of $a_{i_1}a_{i_2}\cdots a_{i_{k-1}}a_{i_1}a_{i_2}\cdots a_{i_{k-1}}$, then $a_{i_1}\neq x$. Let $B=\{e,a_1,a_2,\dots, a_n\}$  and $\mathcal{M}=(S,A,\mu,s_0,F)$ is a DFSA accepting $L$. Denote $m=max\{i|x^i$ is a subword of $a_{i_1}a_{i_2}\cdots a_{i_{k-1}}\}$. Then $m<k-1$.

Let $\mathcal{A}=(Q,A,B,\sigma,q_0,T)$ be a $gsm$, where $Q=S\times\{0,1,\dots,m\}$ is the set of states, $q_0=(s_0,0)$  the initial state, $T=F\times\{0,1,\dots,m\}$ the terminal states and $\sigma$ the partial function from $Q\times A$ to $\mathcal{P}(Q\times B^*)$ defined by the following equations
\begin{align*}
\sigma((s,i),a)&=\{((sa,0),a)\},\ i\in\{0,1,\dots,m\},\ a\neq x,\\
\sigma((s,i),x)&=\{((sx,i+1),x)\},\ i\in\{0,1,\dots,m-1\},\\
\sigma((s,m),x)&=\{((sx,m),xe^{k-2})\},
\end{align*}
where $sb:= \mu(s,b)$ for $s\in S, b\in A$.

Let $K=\eta_\mathcal{A}(L)\cup\{e\}$. Since $L$ is regular, we have $K$ is regular and maps onto $S^1$. Thus $K_=^{\$}=\Delta_K=K_e^\$ $ is regular. In addition,
$$
K_{a_j}^\$=\{(e,a_j)\}\cup\{(\alpha,\alpha a_j)\delta_B^R|\alpha\in K\}\ (a_j\neq x)
$$
and
\begin{eqnarray*}
K_x^\$&=&\{(e,x)\}\cup\{(\alpha,\alpha x)\delta_B^R|\alpha\in K-B^*\{a_{i_1}a_{i_2}\cdots a_{i_{k-1}},x^m,e\}\}\\
&&\cup\{(\alpha,\alpha xe^{k-2})\delta_B^R|\alpha\in (K\cap B^*\{x^m,e\})-\{e\}\}\\
&&\cup(\cup_{t=0}^{m}\{(\alpha x^t(a_{i_1}a_{i_2}\cdots a_{i_{k-1}})^i,\alpha x^{t+i+1})\delta_B^R|\alpha x^t(a_{i_1}a_{i_2}\cdots a_{i_{k-1}})^i\in K,\\
&&\ \ \ \ \ \ \ \ \ \  1\leq i\leq m-t-1,\alpha \in K-B^*\{x\}\})\\
&&\cup(\cup_{t=0}^{m}\{(\alpha x^t(a_{i_1}a_{i_2}\cdots a_{i_{k-1}})^i,\alpha x^m (xe^{k-2})^{t+i+1-m})\delta_B^R|\\
&&\ \ \ \ \ \ \ \ \ \ \alpha x^t(a_{i_1}a_{i_2}\cdots a_{i_{k-1}})^i\in K, i\geq m-t-1,\alpha \in K-B^*\{x\}\})\\
&&\cup\{(\alpha (a_{i_1}a_{i_2}\cdots a_{i_{k-1}})^i,\alpha (xe^{k-2})^{i+1})\delta^R_B|
\alpha (a_{i_1}a_{i_2}\cdots a_{i_{k-1}})^i\in K,\\
&&\ \ \ i\geq 1,\alpha \in K\cap B^*\{e,x^m\}\}
\end{eqnarray*}
are regular. So $(B,K)$ is an automatic structure for $S^1$. Hence $S$ is automatic by Proposition \ref{111}. Since
\begin{eqnarray*}
K'_=&=&\{(e,e)\}\cup\{(\alpha, \alpha)\delta_B^R| \alpha\in K-B^*\{e\}^+\}\\
&&\cup\{(\alpha xe^{k-2}, \alpha xe^j)\delta_B^R|\alpha xe^{k-2}\in K,\ 0\leq j\leq k-2\}\\
&=&\{(e,e)\}\cup\{(\alpha, \alpha)\delta_B^R| \alpha\in K-B^*\{e\}^+\}\\
&&\cup(\Delta_{K-B^*\{xe^{k-2}\}}\cdot\{(xe^{k-2},xe^j)\delta_B^R|0\leq j\leq k-2\})\cap(K\times Pref(K))\delta_B^R
\end{eqnarray*}
is regular, $S^1$ is prefix-automatic.

Case 2. If $x\not\equiv y$, then $S=sgp^+\langle a_1,a_2,\dots,a_n|a_{i_1}a_{i_2}\cdots a_{i_{k-1}}x=xy\rangle$.

Case 2-1. Suppose $a_{i_1}\not\equiv y$. Let $\mathcal{M}=(S,A,\mu,s_0,F)$ be a DFSA accepting $L$.

Case 2-1-1. If $a_{i_1}a_{i_2}\cdots a_{i_{k-1}}x$ is a subword of $a_{i_1}a_{i_2}\cdots a_{i_{k-1}}a_{i_1}a_{i_2}\cdots a_{i_{k-1}}$, then $(A,L)$ is an automatic structure for $S$. It follows that
\begin{eqnarray*}
L_=^{\$}&=&\Delta_L,\ \ \ L_{a_j}^\$=\{(\alpha,\alpha a_j)\delta_A^R|\alpha\in L\}\ (a_j\neq x),\\
L_x^\$&=&\{(\alpha,\alpha x)\delta_A^R|\alpha\in L-A^*\{a_{i_1}a_{i_2}\cdots a_{i_{k-1}}\}\}\\
&& \cup\{(\alpha a_{i_1}a_{i_2}\cdots a_{i_{k-1}},\alpha xy)\delta_A^R|\alpha (a_{i_1}a_{i_2}\cdots a_{i_{k-1}})\in L\}
\end{eqnarray*}
are regular.  Since $L$ is closed under prefix words, $S$ is prefix-automatic.

Case 2-1-2. If $a_{i_1}a_{i_2}\cdots a_{i_{k-1}}x$ is not a subword of $a_{i_1}a_{i_2}\cdots a_{i_{k-1}}a_{i_1}a_{i_2}\cdots a_{i_{k-1}}$, let $\mathcal{A}=(Q,A,B,\sigma,$ $q_0,T)$ be a $gsm$ where
$Q=S\times\{0,1,\dots,m\},\ B=A\cup\{e\},\ q_0=(s_0,0),\
T=F\times\{0,1,\dots,m\},\ m=max\{i|y^i$ is a subword of $a_{i_1}a_{i_2}\cdots a_{i_{k-1}}\}$\ (note that $m\leq k-2)$ and $\sigma$ the partial function from $Q\times A$ to $\mathcal{P}(Q\times B^*)$ defined by the following equations
\begin{align*}
\sigma((s,i),a)&=\{((sa,0),a)\},\ i\in\{0,1,\dots,m\},\ a\neq y,\\
\sigma((s,i),y)&=\{((sy,i+1),y)\},\ i\in\{0,1,\dots,m-1\},\\
\sigma((s,m),y)&=\{((sy,m),ye^{k-2})\},
\end{align*}
where $sb:=\mu(s,b)$ for $s\in S, b\in A$.

Let $K=\eta_\mathcal{A}(L)\cup\{e\}$. Then by the property of $gsm$, $K$ is regular and maps onto $S^1$. Thus $K_=^{\$}=\Delta_K=K_e^{\$}$ is regular. In addition,
\begin{eqnarray*}
K_{a_j}^\$&=&\{(e,a_j)\}\cup\{(\alpha,\alpha a_j)\delta_B^R|\alpha\in K-\{e\}\}\ (a_j\neq x,a_j\neq y),\\
K_x^\$&=&\{(e,x)\}\cup\{(\alpha,\alpha x)\delta_B^R|\alpha\in K-(B^*\{a_{i_1}a_{i_2}\cdots a_{i_{k-1}}\}\cup \{e\})\}\\
&&\cup\{(\alpha(a_{i_1}a_{i_2}\cdots a_{i_{k-1}})^i,\alpha xy^i)\delta_B^R|\alpha(a_{i_1}a_{i_2}\cdots a_{i_{k-1}})^i\in K,1\leq i\leq m,\\
&&\,\,\ \ \alpha\in K-(B^*\{a_{i_1}a_{i_2}\cdots a_{i_{k-1}}\}\cup \{e\})\}\\
&&\cup\{(\alpha(a_{i_1}a_{i_2}\cdots a_{i_{k-1}})^i,\alpha xy^m(ye^{k-2})^{i-m})\delta_B^R|\alpha(a_{i_1}a_{i_2}\cdots a_{i_{k-1}})^i\in K,\\
&&\ \ \ i>m,\alpha\in K-(B^*\{a_{i_1}a_{i_2}\cdots a_{i_{k-1}}\}\cup \{e\})\},\\
K_y^\$&=&\{(e,y)\}\cup\{(\alpha,\alpha y)\delta_B^R|\alpha\in K-B^*\{y^m,e\}\}\\
&&\cup\{(\alpha,\alpha ye^{k-2})\delta_B^R|\alpha\in K\cap B^*\{y^m,e\}-\{e\}\}
\end{eqnarray*}
are regular, so $(B,K)$ is an automatic structure for $S^1$. Hence $S$ is automatic by Proposition \ref{111}. Since
\begin{eqnarray*}
K'_=&=&\{(e,e)\}\cup\{(\alpha, \alpha)\delta_B^R| \alpha\in K-B^*\{e\}^+\}\\
&&\cup\{(\alpha ye^{k-2}, \alpha ye^j)\delta_B^R|\alpha ye^{k-2}\in K,\ 0\leq j\leq k-2\}\\
&=&\{(e,e)\}\cup\{(\alpha, \alpha)\delta_B^R| \alpha\in K-B^*\{e\}^+\}\\
&&\cup(\Delta_{K-B^*\{ye^{k-2}\}}\cdot\{(ye^{k-2},ye^j)\delta_B^R|0\leq j\leq k-2\})\cap(K\times Pref(K))\delta_B^R
\end{eqnarray*}
is regular, $S^1$ is prefix-automatic.

Case 2-2. Let $a_{i_1}\equiv y$ and $S=sgp^+\langle a_1,a_2,\dots,a_n|y^tux=xy\rangle$ where $|y^tux|=k$ and $u(1)\not\equiv y$. Since $a_{i_1}a_{i_2}\cdots a_{i_k}\not\equiv y^{k-1}x$, we have $u\not\equiv\varepsilon.$

Case 2-2-1. If $y^tux$ is a subword of $y^tuy^tu$, then $(A,L)$ is an automatic structure for S, where $L=A^+-A^*\{y^tu\}A^*$. By noting that
\begin{eqnarray*}
L_=&=&\Delta_L,\\
 L_{a_j}^\$&=&\{(\alpha,\alpha a_j)\delta_A^R|\alpha\in L\}\ (a_j\neq x),\\
L_x^\$&=&\{(\alpha,\alpha x)\delta_A^R|\alpha\in L-A^*\{y^tu\}\}\cup\{(\alpha y^tu,\alpha xy)\delta_A^R|\alpha y^tu\in L\}
\end{eqnarray*}
are regular, $S$ is prefix-automatic since $L$ is closed under prefix words.

Case 2-2-2. If $y^tux$ is not a subword of $y^tuy^tu$, then let $\mathcal{M}=(S,A,\mu,s_0,F)$ be a DFSA accepting $L$. Let $m=max\{i|y^i$ is a subword of $y^tuy^tu\}$. Define a $gsm$ $\mathcal{A}=(Q,A,B,\sigma,q_0,T)$, where $Q=S\times\{0,1,\dots,m\},\ B=A\cup\{e\},\ q_0=(s_0,0),\ T=F\times\{0,1,\dots,m\}$ and $\sigma$ the partial function from $Q\times A$ to $\mathcal{P}(Q\times B^*)$ defined by the following equations
\begin{align*}
\sigma((s,i),a)&=\{((sa,0),a)\},\ i\in\{0,1,\dots,m\},\ a\neq y,\\
\sigma((s,i),y)&=\{((sy,i+1),y)\},\ i\in\{0,1,\dots,m-1\},\\
\sigma((s,m),y)&=\{((sy,m),ye^{k-2})\},
\end{align*}
where $sb:=\mu(s,b)$ for $s\in S, b\in A$.

Let $K=\eta_\mathcal{A}(L)\cup\{e\}$. Then $K$ is regular and maps onto $S^1$. Since
\begin{eqnarray*}
K_=^\$&=&\Delta_K=K_e^\$,\\
K_{a_j}^\$&=&\{(e,a_j)\}\cup\{(\alpha,\alpha a_j)\delta_B^R|\alpha\in K-\{e\}\}\ (a_j\neq x, a_j\neq y),\\
K_y^\$&=&\{(e,y)\}\cup\{(\alpha,\alpha y)\delta_B^R|\alpha\in K-B^*\{y^m,e\}\}\\
&&\cup\{(\alpha,\alpha ye^{k-2})\delta_B^R|\alpha\in K\cap (B^*\{y^m,e\})-\{e\}\},\\
K_x^\$&=&\{(e,x)\}\cup\{(\alpha,\alpha x)\delta_B^R|\alpha\in K-B^*\{y^tu\}\}\\
&&\cup\{(\alpha(y^tu)^i,\alpha xy^i)\delta_B^R|\alpha(y^tu)^i\in K,\alpha\in K-B^*\{y^tu,y\},1\leq i\leq m\}\\
&&\cup\{(\alpha(y^tu)^i,\alpha xy^m(ye^{k-2})^{i-m})\delta_B^R|\alpha(y^tu)^i\in K,i>m,\alpha\in K-B^*\{y^tu,y\}\}\\
&&\cup\{(\alpha y^l(y^tu)^i,\alpha y^lxy^i)\delta_B^R|\alpha y^l(y^tu)^i\in K,\alpha\in K-B^*\{y\},\\
&&\ \ \ \ 1\leq l\leq m-t,1\leq i\leq m\}\\
&&\cup\{(\alpha y^l[y^{m-l}(ye^{k-2})^{t-(m-l)}u](y^tu)^i,\alpha y^lxy^{i+1})\delta_B^R|\\
&&\ \ \ \ \alpha y^l[y^{m-l}(ye^{k-2})^{t-(m-l)}u](y^tu)^i\in K,\alpha\in K-B^*\{y\},\\
&&\ \ \ \ m-t\leq l\leq m,0\leq i\leq m-1\}\\
&&\cup\{(\alpha y^l[y^{m-l}(ye^{k-2})^{t-(m-l)}u](y^tu)^i,\alpha y^lxy^m(ye^{k-2})^{i+1-m})\delta_B^R|\\
&&\ \ \ \ \alpha y^l[y^{m-l}(ye^{k-2})^{t-(m-l)}u](y^tu)^i\in K,\alpha\in K-B^*\{y\},\\
&&\ \ \ \ m-t\leq l\leq m,i>m-1\}\\
&&\cup\{(\alpha[(ye^{k-2})^tu](y^tu)^i,\alpha xy^{i+1})\delta_B^R|\alpha[(ye^{k-2})^tu](y^tu)^i\in K,\\
&&\ \ \ \ \alpha\in K\cap B^*\{y^m,e\}-\{e\},0\leq i\leq m-1\}\\
&&\cup\{(\alpha[(ye^{k-2})^tu](y^tu)^i,\alpha xy^m(ye^{k-2})^{i+1-m})\delta_B^R|\alpha[(ye^{k-2})^tu](y^tu)^i\in K,\\
&&\ \ \ \ \alpha\in K\cap B^*\{y^m,e\}-\{e\},i\geq m\}
\end{eqnarray*}
are all regular, $(B,K)$ is an automatic structure for $S^1$. Thus $S$ is automatic by Proposition \ref{111}. Since
\begin{eqnarray*}
K'_=&=&\{(e,e)\}\cup\{(\alpha, \alpha)\delta_B^R| \alpha\in K-B^*\{e\}^+\}\\
&&\cup\{(\alpha ye^{k-2}, \alpha ye^j)\delta_B^R|\alpha ye^{k-2}\in K,\ 0\leq j\leq k-2\}\\
&=&\{(e,e)\}\cup\{(\alpha, \alpha)\delta_B^R| \alpha\in K-B^*\{e\}^+\}\\
&&\cup(\Delta_{K-B^*\{ye^{k-2}\}}\cdot\{(ye^{k-2},ye^j)\delta_B^R|0\leq j\leq k-2\})\cap(K\times Pref(K))\delta_B^R
\end{eqnarray*}
is regular, $S^1$ is prefix-automatic.
\end{proof}

\begin{theorem}\label{3.1.7}
Let $S=sgp^+\langle a,b|a^kb^l=b^l\rangle$, where $k\geq 1$ and $k+l\geq 1$. Then $S$  is prefix-automatic if and only if $l\leq 1$.
\end{theorem}
\begin{proof}
$(\Leftarrow)$ If $l=0$, then by Theorem \ref{3.1.4}, $S$ is prefix-automatic. If $l=1$, then by Theorem \ref{3.1.5}, $S$ is prefix-automatic.

$(\Rightarrow)$ Suppose that $S=sgp^+\langle a,b|a^kb^l=b^l\rangle$ is automatic  for some $l>1$. Then $S^1$ is also automatic by Proposition \ref{111}. Let $B=\{e,a,b\}$. Then there exists $K\subseteq B^+$ such that $(B,K)$ is an automatic structure for $S^1$ with uniqueness by Propositions \ref{cha} and  \ref{uni}.

For any $s,t\in \mathbb{N}$, there exist $\alpha,\beta\in K$ such that $\alpha=(a^k)^s(a^kb)^tb^{l-2},\ \beta=b^{t+l-1}$. Then $(\alpha,\beta)\in K_b^\$ $ and
\begin{eqnarray*}
\alpha&\equiv& \gamma_1a\cdots\gamma_{sk}a\gamma_{sk+1}(a\tau_{11}a\cdots a\tau_{1k}b\tau_{1(k+1)})\cdots(a\tau_{t1}a\cdots a\tau_{tk}b\tau_{t(k+1)})b\xi_1\ldots
b\xi_{l-2},\\
\beta&\equiv&\eta_1b\cdots\eta_{t+l-1}b\eta_{t+l},
\end{eqnarray*}
where $\tau_{tj},\xi_i\in\{e\}^*,\ \eta_t,\eta_{t+1}\cdots\eta_{t+l}\in\{e\}^*,\
\gamma_i,\tau_{ij}\ (i=1,2,\dots,t-1),\eta_1,\eta_{2}\cdots\eta_{t-1}\in\{e,a^k\}^*$.

Denote $\alpha_1\equiv \gamma_1a\gamma_2a\cdots\gamma_{sk}a\gamma_{sk+1}$. Then
$sk\leq|\alpha_1|<N(sk+1)+ks$ and
$t+l-1\leq|\beta|<N(t+l)+t+l-1$, where $N=|S(M(K))|$.

Let $sk>N(t+l)+t+l-1$ and $t+l-1>|S(M(K_b^\$))|$. Then there is a loop $(u_1,u_2)\delta_B^R$ in $(\alpha_1,\beta)\delta_B^R$. If $b\in con(u_2)$, we have $b\not\in con(u_1)$ since $u_1$ is a subword of $\alpha_1$. Suppose $\alpha\equiv w_1u_1w_2$ and
$\beta\equiv w_1'u_2w_2'$. Since $(\alpha,\beta)\delta_B^R\in K_b^\$$ and $(w_1u_1^2w_2,w_1'u_2^2w_2')\delta_B^R\in K_b^\$$, we have $occ(b,w_1u_1^2w_2)\leq occ(b,w_1'u_2^2w_2')-2$, a contradiction. Hence $b\not\in con(u_2)$, so $con(u_2)\subseteq\{e,a\}$ and $\beta\equiv w_1'u_2w_2'=w_1'u_2^{k+1}w_2'\in K$ which contradicts the uniqueness of $K$. Thus $S$ is not automatic.
\end{proof}

\begin{theorem}\label{3.1.9}
Let $S=sgp^{+}\langle a,b\ |\ a^{k}b=ba\rangle\ (k\geq0)$. Then $S$ is prefix-automatic if and only if $k\leq 1$.
\end{theorem}
\begin{proof} $(\Leftarrow)$ If $k\leq1$, then $S$ is prefix-automatic by Theorems \ref{3.1.5} and  \ref{3.1.3}.

$(\Rightarrow)$ Suppose $k\geq2$. Since $S$ is automatic, $S^{1}$ is automatic by Proposition \ref{111}. So,  there exists
an automatic structure $(B,L)$ with uniqueness for $S^{1}$, where $B=\{e,a,b\}$ and $e$ represents the identity of $S^{1}$. Let $N=|S(M(L))|$, $N_==|S(M(L_=^\$))|$, $N_{a}=|S(M(L_{a}^\$))|$, $N_{b}=|S(M(L_{b}^\$))|$, $\widetilde{N}=max\{ N,N_{=},N_{a},N_{b}\}$.

First we claim that $||\alpha|-|\beta||\leq \widetilde{N}$ for any $(\alpha,\beta)\delta_{B}^{R}\in L_{a}^\$\cup L_{b}^\$$. Otherwise there exists $(\alpha,\beta)\delta_{B}^{R}\in L_{a}^\$\cup L_{b}^\$$ such that $||\alpha|-|\beta||> \widetilde{N}$. We can suppose $|\alpha|>|\beta|+\widetilde{N}$ and $\alpha\equiv\alpha_{1}\alpha_{2}$ with $|\alpha_{1}|=|\beta|$. Then $(\alpha_{2},\varepsilon)\delta_{B}^{R}$ contains a subword $u$ that can be pumped in $M(L_{a}^\$)$ or in  $M(L_{b}^\$)$. So either $L_{a}^\$$ or $L_{b}^\$$ contains words of the form $(\widetilde{\alpha}_{1}u^{j}\widetilde{\alpha}_{2},\beta)\delta_{B}^{R}$ with $j\in\mathbb{N}$ where $\widetilde{\alpha}_{1}u\widetilde{\alpha}_{2}\equiv\alpha$. Since $occ(b,\beta)=occ(b,\widetilde{\alpha}_1u\widetilde{\alpha}_2)$, we have $b\not\in con(u)$. If $a\in con(u)$, then there exists $j\geq1$ such that $occ(a,\widetilde{\alpha}_1u^j\widetilde{\alpha}_2)>occ(a,\gamma)$ for any $\gamma\in B^*$ and $\gamma=\beta$, a contradiction. So $u\equiv e^{|u|}$, which contradicts the uniqueness of $L$.

   For $n,i\in\mathbb{N}, n,i\geq1$, let $\beta_{n},\gamma_{n},\beta_{n}^{(i)}\in L$ with $\beta_{n}=a^{k^{n}},\gamma_{n}=b^{n},\beta_{n}^{(i)}=a^{k^{n}}b^{i}$. We have $|\beta|\geq k^{n},|\gamma_{n}|<(n+1)N+n$ and by the above claim
\begin{eqnarray*}
&&||\beta_{n}|-|\beta_{n}^{(1)}||\leq\widetilde{N},\\
&&||\beta_{n}^{(1)}|-|\beta_{n}^{(2)}||\leq\widetilde{N},\\
&& \ \ \ \ \ \ \ \vdots\\
&&||\beta_{n}^{(n-1)}|-|\beta_{n}^{(n)}||\leq\widetilde{N}
\end{eqnarray*}
 and  so $||\beta_{n}|-|\beta_{n}^{(n)}||\leq n\widetilde{N}$. Hence $|\beta_{n}^{(n)}|\geq|\beta_{n}|-n\widetilde{N}\geq k^{n}-n\widetilde{N}$.
   On the other hand, since $(\gamma_{n},\beta_{n}^{(n)})\in L_{a}^\$\cup L_{b}^\$ $, by the claim we have $||\gamma_{n}|-|\beta_{(n)}^n||\leq\widetilde{N}$. However, $|\beta_{n}^{(n)}|-|\gamma_n|>k^n-n\widetilde{N}-(n+1)N-n\rightarrow\infty$, that is, there exists some $n\in \mathbb{N}, n\geq1$ such that $|\beta_{n}^{(n)}|\geq \widetilde{N}+|\gamma_n|$, a contradiction.

This shows that $k\leq1$.
\end{proof}

\subsection{Automaticity of semigroups of one-relator of length $\leq 3$}

In this section, we prove the following theorem.

\begin{theorem}\label{3.2.6}
Let $S=sgp\langle A|u=v\rangle$, where $A=\{a_1,a_2,\dots,a_n\},\ n\geq2,\ u,v\in A^*,\ |v|\leq|u|\leq3$ and $a,b\in A,\ a\neq b$. Then
\begin{enumerate}
\item[(1)]  $S$ is  prefix-automatic if $u=v\not\in \{aba=ba,\ aab=ba, abb=bb\}$. Moreover, if $u=v\in \{aba=ba,\ aab=ba, abb=bb\}$ then $S$ is not automatic.
\item[(2)] $S$ is biautomatic if one of the following holds: (i) $|u|=3,\ |v|=0$, (ii) $|u|=|v|=3$, (iii) $|u|=2$ and $u=v\not\in \{ab=a,\ ab=b\}$. Moreover, if $u=v\in \{ab=a,\ ab=b\}$ then $S$ is not biautomatic.

\end{enumerate}
\end{theorem}

We prove Theorem \ref{3.2.6} step by step.

We fix $A=\{a_1,a_2,\dots,a_n\},\  u,v\in A^*,\ S=sgp\langle A|u=v\rangle$.

\begin{lemma}\label{3.2.1}
Suppose $|u|=|v|=2$. Then $S$ is biautomatic and prefix-automatic.
\end{lemma}
\begin{proof} Case 1. Suppose $a\not\equiv b,\ a,b\in A$. Then $\{ab=cd\}$ is a Gr\"{o}bner-Shirshov basis. Since $|u|=2$, $S$ is biautomatic and prefix-automatic by Theorem \ref{3.1.3}.

 Case 2. Suppose $a\equiv b,c\equiv d$. By Proposition \ref{fre}, we just need to prove $\widetilde{S}=sgp^+\langle a,c|a^2=c^2\rangle$ is biautomatic. Since $\{a^2=c^2,ac^2=c^2a\}$ is a Gr\"{o}bner-Shirshov basis for $\widetilde{S}$, $L=A^+-A^*\{a^2,ac^2\}A^*$ is a normal form of $S$, where $A=\{a,c\}$. Clearly $L_=^\$=\newcommand{\leftexp}[2]{{\vphantom{#2}}^{#1}{#2}}\leftexp{\$}{L}_==\Delta_L$ is regular. Now,
\begin{eqnarray*}
L_a^\$&=&\{(\alpha,\alpha a)\delta_A^R|\alpha\in L-A^*\{a\}\}\\
&&\cup\{(\alpha(ac)^ia,\alpha c^2(ac)^i)\delta_A^R|\alpha\in L-A^*\{ac,a\}, i\geq0\}\\
&=&\Delta_{L-A^*\{a\}}\{(\$,a)\}\\
&&\cup((\Delta_{L-A^*\{ac,a\}}\{(\$,c)(\$,c)\})\odot(\{(a,a)(c,c)\}^*\{(a,\$)\})),\\
L_c^\$&=&\{(\alpha,\alpha c)\delta_A^R|\alpha\in L-A^*\{ac\}\}\\
&&\cup\{(\alpha(ac)^i,\alpha c(ca)^i)\delta_A^R|\alpha\in L-A^*\{ac,a\}, i\geq1\},\\
\leftidx{_a^\$}{L}&=&\{(\alpha,a\alpha)\delta_A^L|\alpha\in L-\{a,c^2\}A^*\}\\
&&\cup\{(a\alpha,c^2\alpha)\delta_A^L|\alpha\in L-\{a,c^2\}A^*\}\\
&&\cup\{((c^2)^i\alpha,(c^2)^ia\alpha)\delta_A^L|\alpha\in L-\{a,c^2\}A^*,i\geq1\}\\
&&\cup\{((c^2)^ia\alpha,(c^2)^{i+1}\alpha)\delta_A^L|\alpha\in L-\{c^2,a\}A^*,i\geq1\},\\
\leftidx{_c^\$}{L}&=&\{(\alpha,c\alpha)\delta_A^L|\alpha\in L\}
\end{eqnarray*}
are regular.

Since $||\alpha|-|\beta||\leq1$ for $(\alpha,\beta)\delta_A^R\in L_a^\$\cup L_c^\$$ and $(\alpha,\beta)\delta_A^L\in\leftidx{_a^\$}{L}\cup \leftidx{_c^\$}L$, we have $\leftidx{^\$}L_a,\leftidx{^\$}L_c$, $\leftidx{_a}L^\$,\leftidx{_c}L^\$$ are regular by Proposition \ref{sam}. Hence $\widetilde{S}$ is biautomatic. Thus, S is biautomatic and   prefix-automatic.
\end{proof}

\begin{lemma}\label{3.2.2}
Suppose $|u|=2,|v|\leq2$. Then $S$ is  prefix-automatic. Moreover, $S$ is biautomatic if and only if $u=v\not\in\{ab=a, ab=b\},\ a,b\in A$.
\end{lemma}
\begin{proof} Case 1. If $|v|=0$, then $S$ is biautomatic and prefix-automatic by Theorem \ref{3.1.4}.

By Proposition \ref{111},  in order to prove that $S$ is biautomatic, it suffices to prove  $\widetilde{S}=sgp^+\langle a_1,a_2,\dots,a_n|u=v\rangle$ is biautomatic.

Case 2. If $|v|=2$, then $S$ is biautomatic and prefix-automatic by Lemma \ref{3.2.1}.

Case 3. If $|v|=1$, suppose $u\equiv ab,v\equiv c$.

 Case 3-1. If $a\not\equiv b$, $c\not\equiv a$ and $c\not\equiv b$, then $\{ab=c\}$ is a Gr\"{o}bner-Shirshov basis in $\widetilde{S}$. So $\widetilde{S}$ is biautomatic and $S$ is prefix-automatic by Theorem \ref{3.1.1}.

 Case 3-2. If $a\not\equiv b$ and $ c\equiv b$, then $\{ab=b\}$ is a Gr\"{o}bner-Shirshov basis in $\widetilde{S}$ and by Theorem \ref{3.1.8}, $\widetilde{S}$ is automatic but not biautomatic, and $S$ is prefix-automatic.

Case 3-3. If $a\not\equiv b $ and $c\equiv a$, then $\widetilde{S}=spg^+\langle a, b| ab=a\rangle$ is isomorphic to $sgp^+\langle a, b| ab=b\rangle^{rev}$ and by Case 3-2, Lemma \ref{revaut} and Theorem \ref{3.1.5}, $\widetilde{S}$ is automatic but not biautomatic, and $S$ is prefix-automatic.



Case 3-4. If $a\equiv b$ and $c\not\equiv a$, then by Proposition \ref{fre}, it is sufficient to prove $S_1=sgp^+\langle a,c|aa=c\rangle$ is biautomatic.

Note that  $\{a^2=c,ac=ca\}$ is a Gr\"{o}bner-Shirshov basis in $S_1$. Let $A_1=\{a, c\}$ and $L=\{c^ia^j|i\geq0, j\in\{0,1\}, i+j\geq1\}$. Then $L$ is a normal form of $S_1$. Clearly,   $L$ and  $L_=^\$=\Delta_L$ are regular. Since
$$L_a^\$=\{(c^i,c^ia)\delta_{A_1}^R|i\geq1\}\cup\{(c^ia,c^{i+1})\delta_{A_1}^R|i\geq 0\}$$
and
$$L_c^\$=\{(c^i, c^{i+1})\delta_{A_1}^R|i\geq 1\}\cup\{(c^ia,c^{i+1}a)\delta_{A_1}^R|i\geq 0\}$$
are regular, $S_1$ is an  automatic semigroup. Let $\widetilde{L}=A^+-A^*\{aa,ac\}A^*$. Then by Proposition \ref{fre}, $(A, \widetilde{L})$ is an automatic structure for $\widetilde{S}$.

For any $x\in A\cup\{\varepsilon\}$,  $(u,v)\delta_{A}^R\in L_x^{\$}$, we have $||u|-|v||\leq 1$. So by Proposition \ref{sam}, $^{\$}L_x$ is regular. Hence $(A, \widetilde{L})$ is a left-right automatic structure for $\widetilde{S}$.

Since $^{\$}_=L=\Delta_L=_=L^{\$}$,
\begin{eqnarray*}
_a^{\$}L&=&\{(c^i, c^ia)\delta_{A_1}^L|i\geq 1\}\cup\{(c^ia,c^{i+1})\delta_{A_1}^L|i\geq 0\}\\
&=&\{(\$, c)\}\{(c, c)\}^*\{(c,a)\}\cup\{(c,c)\}^*\{(a,c)\},\\
_c^{\$}L&=&\{(\alpha, c\alpha)\delta_{A_1}^L|\alpha \in L\}
\end{eqnarray*}
are regular, $S_1$ is a left-left automatic semigroup.

For any $x\in A\cup\{\varepsilon\}$, $(u,v)\delta_{A}^L\in _x^{\$}L$, we have $||u|-|v||\leq 1$ and by Proposition \ref{sam}, $_xL^{\$}$ is regular. Hence $(A, \widetilde{L})$ is a right-left automatic structure for $\widetilde{S}$.

Therefore, $\widetilde{S}$ is  biautomatic and prefix-automatic, and so is $S$.

Case 3-5. If $a\equiv b$ and $c\equiv a$, then by Propositions \ref{fre} and \ref{sam}, it is sufficient to prove that $S_1=sgp^+\langle a |aa=a\rangle$ is biautomatic. Obviously, $S_1=\{a\}$  is biautomatic. Therefore, $\widetilde{S}$ is  biautomatic and prefix-automatic.

The lemma is proved.
\end{proof}

\begin{lemma}\label{3.2.3}
Suppose $|u|=|v|=3$. Then $S$ is biautomatic and prefix-automatic.
\end{lemma}

\begin{proof}
Let $R=\{u=v\},\ u=abc,\ v=xyz,\ a,b,c,x,y,z\in A$ and $\widetilde{S}=sgp^+\langle A| R\rangle$.

$(i)$ Suppose $a,b,c$ are pairwise different. Then $R$ is a Gr\"{o}bner-Shirshov basis in $\widetilde{S}$. Since $|con(abc)|=|u|$, by Theorem \ref{3.1.3}, $S$ is biautomatic and prefix-automatic.

$(ii)$ Suppose $R=\{aac=xyz\}$, where $a\not\equiv c$. Then $R$ is a Gr\"{o}bner-Shirshov basis in $\widetilde{S}$ and
$L=A^{+}-A^{*}\{aac\}A^{*}$
is a normal form of $\widetilde{S}$. Clearly $L$ is regular. Note that
\begin{eqnarray*}
L_=^{\$}&=&\Delta_L=\leftidx{^{\$}}L_=,\\
L_{a_j}^{\$}&=&\{(\alpha,\alpha a_j)\delta_A^{R}|\alpha\in L\}=\Delta_{L}\cdot\{(\$,a_j)\}\ (a_j\not\equiv c),\\
\leftidx{^{\$}_{a_j}}L&=&\{(\alpha,a_j\alpha)\delta_A^{L}|\alpha\in L\}=\{(\$,a_j)\}\cdot\Delta_{L}\ (a_j\not\equiv a)
\end{eqnarray*}
are regular.
Now we prove that $L_c^{\$}$ and $^{\$}_{a}L$ are regular.

Case 1. $x\not\equiv c,\ xy\not\equiv ac$. Then
$$
L_c^{\$}=\{(\alpha,\alpha c)\delta_A^{R}|\alpha\in L-A^{*}\{aa\}\}\cup\{(\alpha aa,\alpha xyz)\delta_A^{R}|\alpha aa\in L\}
$$
is regular,
$$
^{\$}_aL=\{(\alpha,a\alpha)\delta_A^{L}|\alpha\in L-\{ac\}A^{*}\}\cup\{((ac)^{i}\alpha,(xy)^{i}z\alpha)\delta_A^{L}|\alpha \in L-\{ac\}A^{*}\}
$$
 is regular if $z\equiv a$, and
$$
^{\$}_aL=\{(\alpha,a\alpha)\delta_A^{L}|\alpha\in L-\{ac\}A^{*}\}\cup\{(ac\alpha,xyz\alpha)\delta_A^{L}|\alpha \in L\}
$$
 is regular if $z\not\equiv a$.

Case 2. $x\not\equiv c,\ xy\equiv ac$. Then
$$
L_c^{\$}=\{(\alpha,\alpha c)\delta_A^{R}|\alpha \in L-A^{*}\{aa\}\}\cup\{(\alpha a^{i},\alpha xyz^{i-1})\delta_A^{R}|\alpha \in L-A^{*}\{a\},i\geq2\}
$$
 is regular,
$$
^{\$}_aL=\{(\alpha,a\alpha)\delta_A^{L}|\alpha \in L-\{ac\}A^{*}\}\cup\{((ac)^{i}\alpha,(xy)^{i}z\alpha)\delta_A^{L}|\alpha\in L-\{ac\}A^{*},i\geq1\}
$$
 is regular if $z\equiv a$, and
$$
^{\$}_aL=\{(\alpha,a\alpha)\delta_A^{L}|\alpha \in L-\{ac\}A^{*}\}\cup\{(ac\alpha,xyz\alpha)\delta_A^{L}|\alpha \in L\}
$$
 is regular if $z\not\equiv a$.

Case 3. $x\equiv c$. Then
$$
L_c^{\$}=\{(\alpha,\alpha c)\delta_A^{R}|\alpha\in L-A^{*}\{aa\}\}\cup\{(\alpha (aa)^{i},\alpha x(yz)^{i})\delta_A^{R}|\alpha\in L-A^{*}\{aa\},i\geq1\}
$$
 is regular,
$$
^{\$}_aL=\{(\alpha,a\alpha)\delta_A^{L}|\alpha\in L-\{ac\}A^{*}\}\cup\{((ac)^{i}\alpha,(xy)^{i}z\alpha)\delta_A^{L}|\alpha\in L-\{ac\}A^{*},i\geq1\}
$$
 is regular if $z\equiv a$, and
$$
^{\$}_aL=\{(\alpha,a\alpha)\delta_A^{L}|\alpha\in L-\{ac\}A^{*}\}\cup\{(ac\alpha,xyz\alpha)\delta_A^{L}|\alpha\in L\}
$$
 is regular if $z\not\equiv a$.

Therefore, $L_{a_j}^{\$}$ and $^{\$}_{a_j}L$ are regular for each $j\in\{1,2,\dots,n\}$.

For any $(\alpha,\beta)\delta_A^{R}\in L_{a_{j}}^{\$}\  or\ (\alpha,\beta)\delta_A^{L} \in$ $^{\$}_{a_{j}}L \ (j=1,2,\dots,n)$, since $||\alpha|-|\beta||\leq 1$, by Proposition \ref{sam}, we have $^{\$}L_{a_{j}}\ (or\ _{a_{j}}L^{\$})$ is regular, $j=1,2,\dots,n$. Hence, $\widetilde{S}$ is biautomatic and prefix-automatic.

$(iii)$ Suppose $R=\{abb=xyz\}$, where $a\not\equiv b$. Then $R$ is a Gr\"{o}bner-Shirshov basis in $\widetilde{S}$ and
$L=A^{+}-A^{*}\{abb\}A^{*}$
is a normal form of $\widetilde{S}$. Clearly, $L$,
\begin{eqnarray*}
L_=^{\$}&=&\Delta_L=\leftidx{^{\$}}L_=,\\
L_{a_j}^{\$}&=&\{(\alpha,\alpha a_j)\delta_A^{R}|\alpha\in L\}\ (a_j\not\equiv b),\\
^{\$}_{a_j}L&=&\{(\alpha,a_j\alpha)\delta_A^{L}|\alpha\in L\}\ (a_j\not\equiv a)
\end{eqnarray*}
are all regular. Now we prove that $L_b^{\$}$ and $^{\$}_{a}L$ are regular.

Case 1. $x\not\equiv b$. Then
$$
L_b^{\$}=\{(\alpha,\alpha b)\delta_A^{R}|\alpha\in L-A^{*}\{ab\}\}\cup\{(\alpha ab,\alpha xyz)\delta_A^{R}|\alpha \in L\}
$$
 is regular,
$$
^{\$}_aL=\{(\alpha,a\alpha)\delta_A^{L}|\alpha\in L-\{bb\}A^{*}\}\cup\{((bb)^{i}\alpha,(xy)^{i}z\alpha)\delta_A^{L}|\alpha \in L-\{bb\}A^{*}\}
$$
 is regular if $z\equiv a$,
$$
^{\$}_aL=\{(\alpha,a\alpha)\delta_A^{L}|\alpha\in L-\{bb\}A^{*}\}\cup\{(bb\alpha,xyz\alpha)\delta_A^{L}|\alpha\in L\}
$$
 is regular if $z\not\equiv a\ and\ z\not\equiv b$ (or $z\equiv b\ and\ y\not\equiv a$), and
$$
^{\$}_aL=\{(\alpha,a\alpha)\delta_A^{L}|\alpha\in L-\{bb\}A^{*}\}\cup\{(b^{i+2}\alpha,x^{i+1}yz\alpha)\delta_A^{L}|\alpha\in L-\{b\}A^*,i\geq0\}
$$
 is regular if $z\equiv b\ and\ y\equiv a$.

Case 2. $x\equiv b$ and $xy\not\equiv bb$. Then
$$
L_b^{\$}=\{(\alpha,\alpha b)\delta_A^{R}|\alpha \in L-A^{*}\{ab\}\}\cup\{(\alpha (ab)^{i},\alpha x(yz)^{i})\delta_A^{R}|\alpha \in L-A^{*}\{ab\},i\geq1\}
$$
is regular,
$$
^{\$}_aL=\{(\alpha,a\alpha)\delta_A^{L}|\alpha \in L-\{bb\}A^{*}\}\cup\{((bb)^{i}\alpha,(xy)^{i}z\alpha)\delta_A^{L}|\alpha\in L-\{bb\}A^{*},i\geq1\}
$$
 is regular if $z\equiv a$,
$$
^{\$}_aL=\{(\alpha,a\alpha)\delta_A^{L}|\alpha \in L-\{bb\}A^{*}\}\cup\{(bb\alpha,xyz\alpha)\delta_A^{L}|\alpha \in L\}$$
 is regular if $z\not\equiv a\ and\ z\not\equiv b$ (or $z\equiv b\ and\ y\not\equiv a$), and
$$
^{\$}_aL=\{(\alpha,a\alpha)\delta_A^{L}|\alpha \in L-\{bb\}A^{*}\}\cup\{(b^{i+2}\alpha,x^{i+1}yz\alpha)\delta_A^{L}|\alpha \in L-\{b\}A^*,i\geq 0\}
$$
 is regular if  $z\equiv b\ and\ y\equiv a$.

Case 3. $x\equiv b$ and $xy\equiv bb$. Then
\begin{eqnarray*}
 L_b^{\$}&=&\{(\alpha,\alpha b)\delta_A^{R}|\alpha\in L-A^{*}\{ab\}\}\\
 &&\cup\{(\alpha a^{i_1}(ab)^{i_2}\cdots a^{i_{n-2}}(ab)^{i_{n-1}}a^{i_n}(ab),\alpha xyz^{i_1}(zy)^{i_2}\cdots z^{i_{n-2}}(zy)^{i_{n-1}}z^{i_n+1})\delta_A^{R} | \\
 &&\ \ \ \ \alpha\in L-A^{*}\{ab,a\},i_1,i_2,\dots,i_n\geq0,n\in \mathbb{N}\}
 \end{eqnarray*}
 is regular,
$$
^{\$}_aL=\{(\alpha,a\alpha)\delta_A^{L}|\alpha\in L-\{bb\}A^{*}\}\cup\{((bb)^{i}\alpha,(xy)^{i}z\alpha)\delta_A^{L}|\alpha\in L-\{bb\}A^{*},i\geq1\}
$$
 is regular if $z\equiv a$, and
$$
^{\$}_aL=\{(\alpha,a\alpha)\delta_A^{L}|\alpha\in L-\{bb\}A^{*}\}\cup\{(bb\alpha,xyz\alpha)\delta_A^{L}|\alpha \in L\}
$$
 is regular if $z\not\equiv a$.

Therefore, $L_{a_j}^{\$}$ and $^{\$}_{a_j}L$ are regular for each $j\in\{1,2,\dots,n\}$.

 Hence, $\widetilde{S}$ is biautomatic and prefix-automatic.

$(iv)$ Suppose $R=\{aba=aya\}$, where $a\not\equiv b,y\not\equiv b$. Then $R$ is a Gr\"{o}bner-Shirshov basis in $S$. Since $con(aba)\nsubseteq con(aya)$, by Theorem \ref{3.1.3}, $S$ is biautomatic and prefix-automatic.

$(v)$ Suppose $R=\{aba=xyx\}$, where $a\not\equiv b,x\not\equiv a$.

Case 1. $x\not\equiv b$ or $y\not\equiv a$. Then $\{aba=xyx,abxyx=xyxba\}$ is a Gr\"{o}bner-Shirshov basis in $\widetilde{S}$ and
$L=A^{+}-A^{*}\{aba,abxyx\}A^{*}$
is a normal form of $\widetilde{S}$. Clearly, $L$,
\begin{eqnarray*}
L_=^{\$}&=&\Delta_L=\leftidx{^{\$}}L_=,\\
L_{a_j}^{\$}&=&\{(\alpha,\alpha a_j)\delta_A^{R}|\alpha\in L\}\ (a_j\not\equiv a\ and\ a_j\not\equiv x),\\
^{\$}_{a_j}L&=&\{(\alpha,a_j\alpha)\delta_A^{L}|\alpha\in L\}\ (a_j\not\equiv a),\\
L_a^{\$}&=&\{(\alpha,\alpha a)\delta_A^{R}|\alpha\in L-A^{*}\{ab\}\}\\
&&\cup\{(\alpha(abxy)^{i}ab,\alpha xyx(bayx)^{i})\delta_A^{R}|\alpha\in L-A^{*}\{ab,abxy\},i\geq0\},\\
L_x^{\$}&=&\{(\alpha,\alpha x)\delta_A^{R}|\alpha\in L-A^{*}\{abxy\}\}\\
&&\cup\{(\alpha(abxy)^{i},\alpha x(yxba)^{i})\delta_A^{R}|\alpha\in L-A^{*}\{ab,abxy\},i\geq1\},\\
^{\$}_aL&=&\{((bxyx)^{i}ba\alpha,(xyxb)^{i}xyx\alpha)\delta_A^{L}|\alpha\in L-\{ba,bxyx\}A^{*},i\geq0\}\\
&&\cup\{((bxyx)^{i}\alpha,(xyxb)^{i}a\alpha)\delta_A^{L}|\alpha\in L-\{ba,bxyx\}A^{*},i\geq0\}
\end{eqnarray*}
are all regular.

Therefore, $L_{a_j}^{\$}$ and $^{\$}_{a_j}L$ are regular for each $j\in\{1,2,\dots,n\}$.

 Hence, $\widetilde{S}$   is biautomatic and  prefix-automatic.

Case 2. $xy\equiv ba$. Then $\{aba=bab,ab^{i+1}ab=babba^{i}|i\geq1\}$ is a Gr\"{o}bner-Shirshov basis in $\widetilde{S}$ and
$
L=A^{+}-A^{*}\{aba,\ \{ab\}\{b\}^{+}\{ab\}\}A^{*}
$
is a normal form of $\widetilde{S}$. Clearly, $L$,
\begin{eqnarray*}
L_=^{\$}&=&\Delta_L=\leftidx{^{\$}}L_=,\\
L_{a_j}^{\$}&=&\{(\alpha,\alpha a_j)\delta_A^{R}|\alpha\in L\}\ (a_j\not\equiv a\ and\ a_j\not\equiv b),\\
^{\$}_{a_j}L&=&\{(\alpha,a_j\alpha)\delta_A^{L}|\alpha\in L\}\ (a_j\not\equiv a),\\
L_a^{\$}&=&\{(\alpha a^{i_1}(ab^{j_1+1}a)a^{i_2}(ab^{j_2+1}a)\cdots a^{i_m}(ab^{j_m+1}a)a^{i_{m+1}}ab,\ \alpha bab^{i_1}(bba^{j_1+1})b^{i_2}\\
&&(bba^{j_2+1})\cdots b^{i_m}(bba^{j_m+1})b^{i_{m+1}+1})\delta_A^{R}\ |\ \alpha\in L-A^{*}\{a,ab\},\\
&&m\geq0,i_1,\dots,i_{m+1}\geq0,j_1,\dots,j_{m}\geq1\}\\
&&\cup\{(\alpha,\alpha a)\delta_A^{R}|\alpha\in L-A^{*}\{ab\}\},\\
L_b^{\$}&=&\{(\alpha a^{i_1}(ab^{j_1+1}a)a^{i_2}(ab^{j_2+1}a)\cdots a^{i_m}(ab^{j_m+1}a)a^{i_{m+1}}ab^{k+1}a,\ \alpha bab^{i_1}(bba^{j_1+1})\\
&&b^{i_2}(bba^{j_2+1})\cdots b^{i_m}(bba^{j_m+1})b^{i_{m+1}}bba^{k})\delta_A^{R}\ |\ \alpha\in L-A^{*}\{a,ab\},\\
&&  m\geq0,i_1,\dots,i_{m+1}\geq0, j_1,\dots,j_{m}\geq1\}\\
&&\cup\{(\alpha,\alpha b)\delta_A^{R}|\alpha\in L-A^{*}\{ab\}\{b\}^+\{a\}\},\\
\leftidx{_a^\$}{L}&=&\{(\alpha,a\alpha)\delta_A^L|\alpha\in L-\{ba\}A^*-\{b\}\{b\}^+\{ab\}A^*\}\\
&&\cup\{(ba^i\alpha,b^iab\alpha)\delta_A^L|\alpha\in L-\{ba,a\}A^*-\{b\}\{b\}^+\{ab\}A^*,i\geq1\}\\
&&\cup\{(b^{i+1}ab\alpha,babba^i\alpha)\delta_A^L|\alpha\in L-\{ba,a\}A^*-\{b\}^+\{ab\}A^*,i\geq2\}\\
&&\cup\{(b^2ab\alpha,babba\alpha)\delta_A^L|i=1,\alpha\in L-\{a,b\}A^*\}\\
&&\cup\{(b^2ab^n\alpha,b^{n}abba\alpha)\delta_A^L|i=1,\alpha\in L-\{b,ab\}A^*,n\geq 2\}\\
&&\cup\{(b^{i+1}ab^2a^j\alpha,b^2ab^2a^2b^{i-1}a^{j-1}\alpha)\delta_A^L|i\geq3\ or\ i=1,j\geq3,\alpha\in L-\{ba,a\}A^*\\
&&\ \ \ -\{b\}\{b\}^+\{ab\}A^*\}\\
&&\cup\{(b^3ab^2a^j\alpha,b^{j+1}ab^2a^2b\alpha)\delta_A^L|(i=2)j\geq3,\alpha\in L-\{ba,a\}A^*-\{b\}\{b\}^+\{ab\}A^*\}\\
&&\cup\{(b^{i+1}ab^2a\alpha,b^2ab^2a^2b^{i-1}\alpha)\delta_A^L|(j=1)i\geq1,\alpha\in L-\{ba,a\}A^*-\{b\}\{b\}^+\{ab\}A^*\}\\
&&\cup\{(b^3ab^2a^2\alpha,b^3ab^2a^2b\alpha)\delta_A^L|(j=2,i=2)\alpha\in L-\{ba,a\}A^*-\{b\}\{b\}^+\{ab\}A^*\}\\
&&\cup\{(b^2ab^2a^2\alpha,b^2ab^2a^3\alpha)\delta_A^L|(j=2,i=1)\alpha\in L-\{ba,a\}A^*-\{b\}\{b\}^+\{ab\}A^*\}\\
&&\cup\{(b^{i+1}ab^2a^2\alpha,b^2ab^2a^2b^{i-1}a\alpha)\delta_A^L|(j=2)i\geq3,\alpha\in L-\{a,b\}A^*\}\\
&&\cup\{(b^{i+1}ab^2a^2b\alpha,b^3ab^2a^2b^2a^{i-2}\alpha)\delta_A^L|(j=2)i\geq 4,\alpha\in L-\{a\}A^*-\{b\}^+\{ab\}A^*\}\\
&&\cup\{(b^4ab^2a^2b^n\alpha,b^{n+2}ab^2a^2b^2a\alpha)\delta_A^L|(j=2,i=3)\alpha\in L-\{b,ab\}A^*,n\geq1\}
\end{eqnarray*}
are all regular.

Therefore, $L_{a_j}^{\$}$ and $^{\$}_{a_j}L$ are regular for each $j\in\{1,2,\dots,n\}$.

Hence, $\widetilde{S}$   is biautomatic and  prefix-automatic.

$(vi)$ Suppose $R=\{a^{3}=b^{3}\}$, where $a\not\equiv b$. Then
by Proposition \ref{fre}, we just need to prove $\widetilde{S}=sgp^{+}\langle a,b|a^{3}=b^{3}\rangle$ is biautomatic   and prefix-automatic.

Since $\{a^{3}=b^{3},ab^{3}=b^{3}a\}$ is a Gr\"{o}bner-Shirshov basis in $\widetilde{S}$,
$L=A^{+}-A^{*}\{a^{3},ab^{3}\}A^{*}$
is a normal form of $\widetilde{S}$. Clearly, $L$,
\begin{eqnarray*}
L_=^{\$}&=&\Delta_L=\leftidx{^{\$}}L_=,\\
L_a^{\$}&=&\{(\alpha, \alpha a)\delta_A^{R}|\alpha\in L-A^{*}\{a^{2}\}\}\cup\{(\alpha aa,\alpha b^{3})\delta_A^{R}|\alpha\in L-A^{*}\{a\}\},\\
L_b^{\$}&=&\{(\alpha,\alpha b)\delta_A^{R}|\alpha\in L-A^{*}\{ab^{2}\}\}\cup\{(\alpha abb, \alpha bbba)|\alpha\in L-A^{*}\{aa\}\},\\
^{\$}_aL&=&\{(\alpha,a\alpha)\delta_A^{L}|\alpha\in L-\{a^{2},b^{3}\}A^{*}\}\cup\{(aa\alpha,bbb\alpha)\delta_A^{L}|\alpha\in L-\{a,b^3\}A^{*}\}\\
&&\cup\{((b^{3})^{i}\alpha,(b^{3})^{i}a\alpha)\delta_A^{L}|\alpha\in L-\{aa,bbb\}A^{*},i\geq1\}\\
&&\cup\{((b^{3})^{i}a^{2}\alpha,(b^{3})^{i+1}\alpha)\delta_A^{L}|\alpha\in L-\{a,b^{3}\}A^*,i\geq1\},\\
^{\$}_bL&=&\{(\alpha,b\alpha)\delta_A^{L}|\alpha\in L\}
\end{eqnarray*}
are regular. Hence, $\widetilde{S}$   is biautomatic and  prefix-automatic.

Therefore, $S=sgp\langle a_1,a_2,\dots,a_n|abc=xyz\rangle$ is biautomatic and  prefix-automatic.
\end{proof}

\begin{lemma}\label{3.2.4}
Suppose $S=sgp^{+}\langle a, b|aba=ba\rangle$. Then  $S$ is not automatic.
\end{lemma}

\begin{proof}
Since $\{ab^{i}a=b^{i}a|i\geq1\}$ is a Gr\"{o}bner-Shirshov basis in $S$,  we have $L=A^{+}-A^{*}\{a\}\{b\}^{+}\{a\}A^{*}$ is a normal form of $S$, where $A=\{a,b\}$.

Suppose $S$ is automatic. Then, by Proposition \ref{111}, $S^{1}$ is automatic.

Denote $B=\{e,a,b\}$. Then there exists an automatic structure $(B,K)$ for $S^{1}$ with uniqueness.

For any $ i,j\geq1$, there exist $\alpha,\beta\in K$, such that $\alpha=a^{j}b^{i},\beta=b^{i}a$. Since $a^{j}b^{i}\cdot a=b^{i}a$, we have $(\alpha,\beta)\delta_B^{R}\in K_a^{\$}$.

Since $\alpha,\beta\in K$, we have
\begin{eqnarray*}
\alpha&\equiv& \gamma_1a\gamma_2a\cdots\gamma_ja\gamma_1'b\gamma_2'b\cdots\gamma_i'b\gamma_{i+1}',\\
\beta&\equiv&\tau_1b\tau_2b\cdots\tau_ib\tau_{i+1}a\tau_{i+2},
\end{eqnarray*}
where $\gamma_1,\dots,\gamma_j,\gamma_1',\dots,\gamma_{i+1}'\in\{e\}^{*}$, $\tau_1,\dots,\tau_i\in\{e,a\}^{*}$ and $\tau_{i+1},\tau_{i+2}\in\{e\}^{*}$.

Let $N=|S(M(K))|$, $\alpha_1\equiv \gamma_1a\gamma_2a\cdots\gamma_ja$ and $\beta_1\equiv\tau_1b\tau_2b\cdots\tau_ib$. Then $j\leq|\alpha_1|<jN+j$ and $i\leq|\beta_1|<iN+i$.

Choose $j>iN+i,\ i>|S(M(K_a^{\$}))|$. Then there exists a loop $(u,v)\delta_B^{R}$ in $(\alpha_1,\beta_1)\delta_B^{R}$.
Suppose $\alpha\equiv\alpha'u\alpha^{''},\beta\equiv\beta'v\beta^{''}$.

If $b\not\in con(v)$, then $con(v)\in\{e,a\}$. By the relations in $S$, $\beta'\beta^{''}=\beta$ and $\beta'\beta^{''},\beta\in K$ which contradicts the uniqueness of $K$.

If $b\in con(v)$, we have $occ(b,\alpha'\alpha''a)\neq occ(b,\beta'\beta'')$ since $b\not\in con(u)$. But $(\alpha'\alpha^{''},\beta'\beta'')\delta_B^{R}\in K_a^{\$}$,  a contradiction.

Therefor, $S^{1}$ is not automatic.
\end{proof}

\begin{lemma}\label{3.2.5}
Let $S=sgp\langle A|u=v\rangle$, where $A=\{a_1,a_2,\dots,a_n\}$ $(n\in \mathbb{N})$, $u,v\in A^*$ and $|v|\leq|u|=3$. Then
\begin{enumerate}
\item[(i)] if $u=v \in\{aba=ba,aab=ba,abb=bb| a,b\in A,a\not\equiv b\}$, then $S$ is not automatic;

\item[(ii)] if $u=v \not\in\{aba=ba,aab=ba,abb=bb| a,b\in A,a\not\equiv b\}$, then $S$ is  prefix-automatic;

\item[(iii)] if $|v|=0$ or 3, then S is biautomatic.

\end{enumerate}
\end{lemma}

\begin{proof}
$(i)$ If $u=v \in\{aba=ba,aab=ba,abb=bb| a,b\in A,\ a\not\equiv b\}$, then  by Proposition \ref{111}, Theorems \ref{3.1.7} and \ref{3.1.9}, and Lemma \ref{3.2.4}, $S=sgp\langle A|u=v\rangle$ is not automatic.

Now we prove $(ii)$ and $(iii)$. Suppose $u=v \not\in\{aba=ba,aab=ba,abb=bb| a,b\in A,a\not\equiv b\}$ and $\widetilde{S}=sgp^+\langle A|u=v\rangle$.

$1)$ Suppose $|v|=0$. If $u=v\in\{aaa=\varepsilon,\ aab=\varepsilon, abb=\varepsilon, abc=\varepsilon| a,b\in A,\ a\not\equiv b\}$, then by Theorem \ref{3.1.4}, $S$ is biautomatic and prefix-automatic.

Let $\widetilde{S}=sgp^+\langle e, a, b| aba=e,ea=ae=a,eb=eb=b,ee=e\rangle$. If $u=aba$, then $\{aba=e, aab=e, ba=ab, ae=a,ea=a, be=b,eb=b, ee=e\}$ is a Gr\"{o}bner-Shirshov basis in $\widetilde{S}$. Let $B=\{e, a, b\}$, $L=B^+-B^*\{aba, aab, ba, ae,$ $ea, be, eb, ee\}B^*$. Then L is a  normal form of $\widetilde{S}$. Then $L_==L_e=_{e}L=\Delta_L$,
\begin{eqnarray*}
L_a^{\$}&=&\{(e,a)\}\cup\{(\alpha, \alpha a)\delta_B^R|\alpha\in L-B^*\{b,e\}\}\\
&&\cup\{(\alpha ab, \alpha)\delta_B^R|\alpha ab\in L\}\cup\{(\alpha b, \alpha ab)\delta_B^R|\alpha\in L-B^*\{a\}\},\\
L_b^{\$}&=&\{(e,b)\}\cup\{(\alpha, \alpha b)\delta_B^R|\alpha\in L-B^*\{aa, e\}\}\\
&&\cup\{(\alpha aa, \alpha)\delta_B^R|\alpha aa\in L\}
\end{eqnarray*}
are regular and so $\widetilde{S}$ is prefix-automatic. Noting that
\begin{eqnarray*}
_a^{\$}L&=&\{(e,a)\}\cup\{(\alpha, a\alpha)\delta_B^L|\alpha\in L-\{ab,ba,e\}B^*\}\\
&&\cup\{(ab\alpha, \alpha)\delta_B^L|ab\alpha\in L\}\cup\{(ba\alpha, \alpha)\delta_B^L|ba\alpha\in L\},\\
_b^{\$}L&=&\{(e,b)\}\cup\{(\alpha, b\alpha)\delta_B^L|\alpha\in L-\{a,e\}B^*\}\\
&&\cup\{(a\alpha, ab\alpha)\delta_B^L| a\alpha\in L\ and\ \alpha\in B^*-\{a\}B^*\}\\
&&\cup\{(aa\alpha, \alpha)\delta_B^L|aa\alpha\in L\}
\end{eqnarray*}
are regular, by Proposition \ref{sam}, $^{\$}L_a,\ ^{\$}L_b,\ _aL^{\$},\ _bL^{\$}$ are regular and hence $\widetilde{S}$ is biautomatic. This shows that $S$ is biautomatic and prefix-automatic.

$2)$ Suppose $|v|=1$. If $u=v\in\{abc=x,aab=x,abb=x,aba=a,aaa=a|a,b,c,x\in A,a\not\equiv b,a\not\equiv c,b\not\equiv c\}$, then $\{u=v\}$ is a Gr\"{o}bner-Shirshov basis in $\widetilde{S}$. Then by Theorem \ref{3.1.5}, $S$ is prefix-automatic.

If $u\equiv aba,\ v\equiv x$, where $x\in A-\{a\}$, then $\{aba=x,abx=xba\}$ is a Gr\"{o}bner-Shirshov basis in $\widetilde{S}$ and so $L=A^{+}-A^{*}\{aba,abx\}A^{*}$ is a normal form of $\widetilde{S}$. Clearly, $L$,
\begin{eqnarray*}
L_=^{\$}&=&\Delta_L,\ \ \ 
L_{a_j}^{\$}=\{(\alpha,\alpha a_j)\delta_A^{R}|\alpha\in L\}\ (a_j\not\equiv a\ and\ a_j\not\equiv x),\\
L_{a}^{\$}&=&\{(\alpha,\alpha a)\delta_A^{R}|\alpha\in L-A^{*}\{ab\}\}\cup\{(\alpha ab,\alpha x)\delta_A^{R}|\alpha\in L-A^{*}\{ab\}\},\\
L_{x}^{\$}&=&\{(\alpha,\alpha x)\delta_A^{R}|\alpha\in L-A^{*}\{ab\}\}\cup\{(\alpha ab,\alpha xba)\delta_A^{R}|\alpha\in L-A^{*}\{ab\}\}\ (x\not\equiv b),\\
L_{x}^{\$}&=&\{(\alpha,\alpha x)\delta_A^{R}|\alpha\in L-A^{*}\{ab\}\}\cup\{(\alpha a^{i}b,\alpha xba^{i})\delta_A^{R}|\alpha\in L-A^{*}\{a\},i\geq1\}\ (x\equiv b)
\end{eqnarray*}
are regular.

Therefore, $L_{a_j}^{\$}$  are regular for each $j\in\{1,2,\dots,n\}$.
Hence $\widetilde{S}$ is automatic. Thus $S$ is prefix-automatic.


If $u\equiv a^{3},\ v\equiv x$, where $x\in A-\{a\}$, then $\{aaa=x,ax=xa\}$ is a Gr\"{o}bner-Shirshov basis in $\widetilde{S}$ and $L=A^{+}-A^{*}\{aaa,ax\}A^{*}$ is a normal form of $\widetilde{S}$. Clearly, $L$,
\begin{eqnarray*}
L_=^{\$}&=&\Delta_L,\\
L_{a_j}^{\$}&=&\{(\alpha,\alpha a_j)\delta_A^{R}|\alpha\in L\}\ (a_j\not\equiv a\ and\ a_j\not\equiv x),\\
L_a^{\$}&=&\{(\alpha,\alpha a)\delta_A^{R}|\alpha\in L-A^{*}\{aa\}\}\cup\{(\alpha aa,\alpha x)\delta_A^{R}|\alpha\in L-A^{*}\{a\}\},\\
L_x^{\$}&=&\{(\alpha,\alpha x)\delta_A^{R}|\alpha\in L-A^{*}\{a\}\}\cup\{(\alpha a,\alpha xa)\delta_A^{R}|\alpha\in L-A^{*}\{a\}\}\\
&&\cup\{(\alpha aa,\alpha xaa)\delta_A^{R}|\alpha\in L-A^{*}\{a\}\}
\end{eqnarray*}
are regular.


Therefore, $L_{a_j}^{\$}$  are regular for each $j\in\{1,2,\dots,n\}$.
Hence $\widetilde{S}$ is prefix-automatic.

$3)$ Suppose $|v|=2$.

Case 1. $\widetilde{S}=sgp^{+}\langle A|abc=xy\rangle$, where $a,b,c,x,y\in A$ and $a,b,c$ are pairwise different.
Obviously, $\{abc=xy\}$ is a Gr\"{o}bner-Shirshov basis in $\widetilde{S}$.

Case 1-1. If $xy\not\equiv bc$, by Theorem \ref{3.1.6}, $\widetilde{S}$ is prefix-automatic.

Case 1-2. If $xy\equiv bc$, by Proposition \ref{fre}, it is sufficient to prove $\widetilde{S}=sgp^{+}\langle a,b,c|abc=bc\rangle$ is prefix-automatic.

Since $\{abc=bc\}$ is a Gr\"{o}bner-Shirshov basis in $\widetilde{S}$,  we have $L=B^{+}-B^{*}\{abc\}B^{*}$ is a normal form of $\widetilde{S}$, where  $B=\{a,b,c\}$. Clearly,
\begin{eqnarray*}
&&L,\ \ L_=^{\$}=\Delta_L,\ \ \ L_a^{\$}=\{(\alpha,\alpha a)\delta_B^{R}|\alpha\in L\},\ \ \ \  L_b^{\$}=\{(\alpha,\alpha b)\delta_B^{R}|\alpha\in L\},\\
&&L_c^{\$}=\{(\alpha,\alpha c)\delta_B^{R}|\alpha\in L-B^{*}\{ab\}\}\cup\{(\alpha a^{i}b,\alpha bc)\delta_B^{R}|\alpha\in L-B^{*}\{a\},i\geq1\}
\end{eqnarray*}
are regular.
Hence $\widetilde{S}$ is prefix-automatic.

Case 2. $\widetilde{S}=sgp^{+}\langle A|aab=xy\rangle$, where $a,b,x,y\in A$ and $a\not\equiv b$. Obviously, $\{aab=xy\}$ is a Gr\"{o}bner-Shirshov basis in $\widetilde{S}$.

Case 2-1. If $xy\not\equiv ab$, we have $v\not\equiv ba$ since $u=v\not\in \{aba=ba,aab=ba,abb=bb| a,b\in A,a\not\equiv b\}$. By Theorem \ref{3.1.6}, $S$ is  prefix-automatic.

Case 2-2. If $xy\equiv ab$, by Proposition \ref{fre}, it is sufficient to prove $\widetilde{S}=sgp^{+}\langle a,b|aab=ab\rangle$ is prefix-automatic.

Since $\{aab=ab\}$ is a Gr\"{o}bner-Shirshov basis in $\widetilde{S}$, we have $L=B^{+}-B^{*}\{aab\}B^{*}$ is a normal form of $\widetilde{S}$, where $B=\{a,b\}$. Clearly, $L$,
$L_=^{\$}=\Delta_L$,
$L_a^{\$}=\{(\alpha,\alpha a)\delta_B^{R}|\alpha\in L\}$ and
$L_b^{\$}=\{(\alpha,\alpha b)\delta_B^{R}|\alpha\in L-B^{*}\{aa\}\}\cup\{(\alpha a^{i+2},\alpha ab)\delta_B^{R}|\alpha\in L-B^{*}\{a\},i\geq0\}$
are all regular.
Hence $\widetilde{S}$ is  prefix-automatic.

Case 3.  $\widetilde{S}=sgp^{+}\langle A|abb=xy\rangle$, where $a,b,x,y\in A$ and $a\not\equiv b$. Obviously, $\{abb=xy\}$ is a Gr\"{o}bner-Shirshov basis in $\widetilde{S}$.

Since $u=v\not\in \{aba=ba,aab=ba,abb=bb| a,b\in A,a\not\equiv b\}$, we have $xy\not\equiv bb$. By Theorem \ref{3.1.6}, $S$ is  prefix-automatic.

Case 4.  $\widetilde{S}=sgp^{+}\langle A|aba=xy\rangle$, where $a,b,x,y\in A$ and $a\not\equiv b$.

Case 4-1. If $xy\equiv aa$, then $\{aba=aa\}$ is a Gr\"{o}bner-Shirshov basis in $\widetilde{S}$. By Theorem \ref{3.1.6}, $S$ is  prefix-automatic.

Case 4-2. If $xy\equiv ab$, then $\{ab^{i}a=ab^{i}|i\geq1\}$ is a Gr\"{o}bner-Shirshov basis in $\widetilde{S}$. By Proposition \ref{fre}, it is sufficient to prove $\widetilde{S}=sgp^{+}\langle a,b|aba=ab\rangle$ is prefix-automatic.

Let $B=\{a,b\}$, $L=B^{+}-B^{*}\{a\}\{b\}^{+}\{a\}B^{*}$. Then $L$ is a normal form of $\widetilde{S}$. It is clear that $L$,
$L_=^{\$}=\Delta_L$,
$$
L_a^{\$}=\{(\alpha,\alpha a)\delta_B^{R}|\alpha\in L-B^{*}\{a\}\{b\}^{+}\}\cup\{(\alpha ab^{i},\alpha ab^{i})\delta_B^{R}|\alpha\in L-B^{*}\{a\}\{b\}^{+}, i\geq1\}
$$
and
$L_b^{\$}=\{(\alpha,\alpha b)\delta_B^{R}|\alpha\in L\}$
are regular.
Hence $\widetilde{S}$ is  prefix-automatic.

Case 4-3. If $x\equiv a, \ y\not\equiv a$ and $y\not\equiv b$, then $\{ay^{i}ba=ay^{i+1}|i\geq0\}$ is a Gr\"{o}bner-Shirshov basis in $S$. By Proposition \ref{fre}, it is sufficient to prove $\widetilde{S}=sgp^{+}\langle a,b,y|aba=ay\rangle$ is prefix-automatic.

Let $B=\{a,b,y\}$, $L=B^{+}-B^{*}\{a\}\{y\}^{*}\{ba\}B^{*}$. Then $L$ is a normal form of $\widetilde{S}$. Clearly
\begin{eqnarray*}
L_=^{\$}&=&\Delta_L,\ \ \  L_b^{\$}=\{(\alpha,\alpha b)\delta_B^{R}|\alpha\in L\},\ \ \ L_y^{\$}=\{(\alpha,\alpha y)\delta_B^{R}|\alpha\in L\},\\
L_a^{\$}&=&\{(\alpha,\alpha a)\delta_B^{R}|\alpha\in L-B^{*}\{a\}\{y\}^{*}\{b\}\}\\
&&\cup\{(\alpha ay^{i}b,\alpha ay^{i+1})\delta_B^{R}|\alpha\in L-B^{*}\{a\}\{y\}^{*}\{b\}, i\geq0\}
\end{eqnarray*}
are regular.
Hence $\widetilde{S}$ is  prefix-automatic.

Case 4-4. If $xy\equiv ba$, then $u=v \in\{aba=ba,aab=ba,abb=bb| a,b\in A,a\neq b\}$ and hence $S$ is not automatic by $(i)$.

Case 4-5. If $xy\equiv bb$, then $\{aba=bb,ab^{3}=b^{3}a\}$ is a Gr\"{o}bner-Shirshov basis in $\widetilde{S}$. By Proposition \ref{fre}, it is sufficient to prove $\widetilde{S}=sgp^{+}\langle a,b|aba=bb\rangle$ is prefix-automatic.

Let $B=\{a,b\}$, $L=B^{+}-B^{*}\{aba,ab^{3}\}B^{*}$. Then $L$ is a normal form of $\widetilde{S}$. Clearly,
\begin{eqnarray*}
L_=^{\$}&=&\Delta_L,\\
L_a^{\$}&=&\{(\alpha,\alpha a)\delta_B^{R}|\alpha\in L-B^{*}\{ab\}\}\cup\{(\alpha aab,\alpha abb)\delta_B^{R}|\alpha\in L-B^{*}\{ab\}\}\\
&&\cup\{(\alpha bab,b^{3}\alpha)\delta_B^{R}|\alpha\in L-B^{*}\{a,ab^2\}\},\\
L_b^{\$}&=&\{(\alpha,\alpha b)\delta_B^{R}|\alpha\in L-B^{*}\{ab^{2}\}\}\cup\{(\alpha ab^{2},b^{3}\alpha a)\delta_B^{R}|\alpha\in L-B^{*}\{ab\}\}
\end{eqnarray*}
are regular.
Hence $\widetilde{S}$ is prefix-automatic.

Case 4-6. If $x\equiv b,\ y\not\equiv a$ and $y\not\equiv b$, then $\{aba=by,ab^{2}y=byba\}$ is a Gr\"{o}bner-Shirshov basis in $S$. By Proposition \ref{fre}, it is sufficient to prove $\widetilde{S}=sgp^{+}\langle a,b,y|aba=by\rangle$ is prefix-automatic.

Let $B=\{a,b,y\}$, $L=B^{+}-B^{*}\{aba,ab^{2}y\}B^{*}$. Then $L$ is a normal form of $\widetilde{S}$. Obviously,
\begin{eqnarray*}
L_=^{\$}&=&\Delta_L,\\
L_a^{\$}&=&\{(\alpha,\alpha a)\delta_B^{R}|\alpha\in L-B^{*}\{ab\}\}\cup\{(\alpha ab,\alpha by)\delta_B^{R}|\alpha\in L-B^{*}\{ab\}\},\\
L_b^{\$}&=&\{(\alpha,\alpha b)\delta_B^{R}|\alpha\in L\}\},\\
L_y^{\$}&=&\{(\alpha,\alpha y)\delta_B^{R}|\alpha\in L-B^{*}\{ab^{2}\}\}\cup\{(\alpha ab^{2},\alpha byba)\delta_B^{R}|\alpha\in L-B^{*}\{ab\}\}
\end{eqnarray*}
are regular. Hence $\widetilde{S}$ is prefix-automatic.

Case 4-7. If $x\not\equiv a,\ x\not\equiv b$ and $y\equiv a$, then $\{abx^{i}a=x^{i+1}a|i\geq0\}$ is a Gr\"{o}bner-Shirshov basis in $S$. By Proposition \ref{fre}, it is sufficient to prove $\widetilde{S}=sgp^{+}\langle a,b,x|aba=xa\rangle$ is prefix-automatic.

Let $B=\{a,b,x\}$, $L=B^{+}-B^{*}\{ab\}\{x\}^{*}\{a\}B^{*}$. Then $L$ is a normal form of $\widetilde{S}$. Clearly,
$L_=^{\$}=\Delta_L$,
$$
L_a^{\$}=\{(\alpha,\alpha a)\delta_B^{R}|\alpha\in L-B^{*}\{ab\}\{x\}^{*}\}\cup\{(\alpha abx^{i},\alpha x^{i+1}a)\delta_B^{R}|\alpha\in L-B^{*}\{ab\}\{x\}^*\},
$$
$L_b^{\$}=\{(\alpha,\alpha b)\delta_B^{R}|\alpha\in L\}\}$ and $L_x^{\$}=\{(\alpha,\alpha x)\delta_B^{R}|\alpha\in L\}\}$ are all regular.
Hence $\widetilde{S}$ is prefix-automatic.

Case 4-8. If $x\not\equiv a,\ x\not\equiv b$ and $y\not\equiv a,x$, then $\{aba=xy,abxy=xyba\}$ is a Gr\"{o}bner-Shirshov basis in $S$. Let $B=\{a,b,x,y\}$ if $y\not\equiv b$ (Otherwise, let $B=\{a,b,x\}$). By Proposition \ref{fre}, it is sufficient to prove $\widetilde{S}=sgp^{+}\langle B|aba=xy\rangle$ is prefix-automatic. Note that
$L=B^{+}-B^{*}\{aba,abxy\}B^{*}$ is a normal form of $\widetilde{S}$. Clearly,
\begin{eqnarray*}
L_=^{\$}&=&\Delta_L,\\
L_a^{\$}&=&\{(\alpha,\alpha a)\delta_B^{R}|\alpha\in L-B^{*}\{ab\}\}\cup\{(\alpha ab,\alpha xy)\delta_B^{R}|\alpha\in L-B^{*}\{ab\}\},\\
L_b^{\$}&=&\{(\alpha,\alpha b)\delta_B^{R}|\alpha\in L\}\}(y\not\equiv b),\\
L_b^{\$}&=&\{(\alpha,\alpha b)\delta_B^{R}|\alpha\in L-B^*\{abx\}\}\}
\cup\{(\alpha abx,\alpha xyba)\delta_B^{R}|\alpha\in L-B^*\{ab\}\}(y\equiv b),\\
L_x^{\$}&=&\{(\alpha,\alpha x)\delta_B^{R}|\alpha\in L\}\},\\
L_y^{\$}&=&\{(\alpha,\alpha y)\delta_B^{R}|\alpha\in L-B^{*}\{abx\}\}\cup\{(\alpha abx,\alpha xyba)\delta_B^{R}|\alpha\in L-B^{*}\{ab\}\}
\end{eqnarray*}
are all regular.
Hence $\widetilde{S}$ is prefix-automatic.

Case 4-9. If $x\not\equiv a,\ x\not\equiv b$ and $y\equiv x$, then $\{aba=xx,abx^{2}=x^{2}ba\}$ is a Gr\"{o}bner-Shirshov basis in $S$. By Proposition \ref{fre}, it is sufficient to prove $\widetilde{S}=sgp^{+}\langle a,b,x|aba=x^{2}\rangle$ is prefix-automatic.

Let $B=\{a,b,x\}$, $L=B^{+}-B^{*}\{aba,abx^{2}\}B^{*}$. Then $L$ is a normal form of $\widetilde{S}$. Noting that
\begin{eqnarray*}
L_=^{\$}&=&\Delta_L,\\
L_b^{\$}&=&\{(\alpha,\alpha b)\delta_B^{R}|\alpha\in L\}\},\\
L_a^{\$}&=&\{(\alpha,\alpha a)\delta_B^{R}|\alpha\in L-B^{*}\{ab\}\}\\
&&\cup\{(\alpha (abx)^{i}ab,\alpha xx(bax)^{i})\delta_B^{R}|\alpha\in L-B^{*}\{ab,abx\},i\geq0\},\\
L_x^{\$}&=&\{(\alpha,\alpha x)\delta_B^{R}|\alpha\in L-B^{*}\{abx\}\}\\
&&\cup\{(\alpha (abx)^{i},\alpha x(xba)^{i})\delta_B^{R}|\alpha\in L-B^{*}\{abx,ab\},i\geq1\}
\end{eqnarray*}
are regular,
$\widetilde{S}$ is  prefix-automatic.

Case 5.  $S=sgp^{+}\langle A|a^{3}=xy\rangle$, where $a,x,y\in A$.

Case 5-1. If $x\not\equiv a$ and $y\not\equiv a$, then $\{a^{3}=xy,axy=xya\}$ is a Gr\"{o}bner-Shirshov basis in $S$. Let $B=\{a,x,y\}$ if $x\not\equiv y$ (Otherwise, let $B=\{a,x\}$). By Proposition \ref{fre}, it is sufficient to prove $\widetilde{S}=sgp^{+}\langle B|aaa=xy\rangle$ is prefix-automatic.

Note that $L=B^{+}-B^{*}\{aaa,axy\}B^{*}$ is a normal form of $\widetilde{S}$. Clearly,
\begin{eqnarray*}
L_=^{\$}&=&\Delta_L,\\
L_a^{\$}&=&\{(\alpha,\alpha a)\delta_B^{R}|\alpha\in L-B^{*}\{aa\}\}\cup\{(\alpha aa,\alpha xy)\delta_B^{R}|\alpha\in L-B^{*}\{a\}\},\\
L_x^{\$}&=&\{(\alpha,\alpha x)\delta_B^{R}|\alpha\in L-B^*\{ax\}\}\cup\{(\alpha a^ix,\alpha xxa^i)\delta_B^R|\alpha\in L-B^*\{a\},i=1,2\},\ x\equiv y,\\
L_x^{\$}&=&\{(\alpha,\alpha x)\delta_B^{R}|\alpha\in L\}\},\   x\not\equiv y, \\
L_y^{\$}&=&\{(\alpha,\alpha y)\delta_B^{R}|\alpha\in L-B^{*}\{ax\}\}\cup\{(\alpha a^{i}x,\alpha xya^{i})\delta_B^{R}|\alpha\in L-B^{*}\{a\},i=1,2\}
\end{eqnarray*}
are regular.
Hence $\widetilde{S}$ is prefix-automatic.

Case 5-2. If $x\not\equiv a$ and $y\equiv a$, then $\{a^{3}=xa,ax^{i}a=x^{i}a^{2}|i\geq1\}$ is a Gr\"{o}bner-Shirshov basis in $S$. By Proposition \ref{fre}, it is sufficient to prove $\widetilde{S}=sgp^{+}\langle a,x|aaa=xa\rangle$ is prefix-automatic.

Let $B=\{a,x\}$, $L=B^{+}-B^{*}\{aaa\}B^{*}-B^{*}\{a\}\{x\}^{+}\{a\}B^{*}$. Then $L$ is a normal form of $\widetilde{S}$. Clearly,
\begin{eqnarray*}
L_=^{\$}&=&\Delta_L,\\
L_x^{\$}&=&\{(\alpha,\alpha x)\delta_B^{R}|\alpha\in L\}\},\\
L_a^{\$}&=&\{(\alpha,\alpha a)\delta_B^{R}|\alpha\in L-B^{*}\{aa\}-B^{*}\{a\}\{x\}^{+}\}\\
&&\cup\{(\alpha aa,\alpha xa)\delta_B^{R}|\alpha\in L-B^{*}\{a\}-B^{*}\{a\}\{x\}^{+}\}\\
&&\cup\{(\alpha ax^{i},\alpha x^{i}a^{2})\delta_B^{R}|\alpha\in L-B^{*}\{a\}-B^{*}\{a\}\{x\}^{+},i\geq1\}\\
&&\cup\{(\alpha aax^{i},\alpha x^{i+1}a)\delta_B^{R}|\alpha\in L-B^{*}\{a\}-B^{*}\{a\}\{x\}^{+},i\geq1\}
\end{eqnarray*}
are regular.
Hence $\widetilde{S}$ is  prefix-automatic.

Case 5-3. If $x\equiv a$ and $y\not\equiv a$, then $\{a^{3}=ay,ay^{i}a=a^{2}y^{i}|i\geq1\}$ is a Gr\"{o}bner-Shirshov basis in $S$. By Proposition \ref{fre}, it is sufficient to prove $\widetilde{S}=sgp^{+}\langle a,y|aaa=ay\rangle$ is prefix-automatic.

Let $B=\{a,y\}$, $L=(B^{+}-B^{*}\{aaa\}B^{*})-B^{*}\{a\}\{y\}^{+}\{a\}B^{*}$. Then $L$ is a normal form of $\widetilde{S}$. By noting that
\begin{eqnarray*}
L_=^{\$}&=&\Delta_L,\\
L_y^{\$}&=&\{(\alpha,\alpha y)\delta_B^{R}|\alpha\in L\},\\
L_a^{\$}&=&\{(\alpha,\alpha a)\delta_B^{R}|\alpha\in L-B^{*}\{aa\}-B^{*}\{a\}\{y\}^{+}\}\\
&&\cup\{(\alpha aa,\alpha ay)\delta_B^{R}|\alpha\in L-B^{*}\{a\}-B^{*}\{a\}\{y\}^{+}\}\\
&&\cup\{(\alpha ay^{i},\alpha a^{2}y^{i})\delta_B^{R}|\alpha\in L-B^{*}\{a\}-B^{*}\{a\}\{y\}^{+},i\geq1\}\\
&&\cup\{(\alpha aay^{i},\alpha ay^{i+1})\delta_B^{R}|\alpha\in L-B^{*}\{a\}-B^{*}\{a\}\{y\}^{+},i\geq1\}
\end{eqnarray*}
are regular,  $\widetilde{S}$ is prefix-automatic.

Case 5-4. If $x\equiv a$ and $y\equiv a$, then $\{a^{3}=a^{2}\}$ is a Gr\"{o}bner-Shirshov basis in $S$. By Proposition \ref{fre}, it is sufficient to prove $\widetilde{S}=sgp^{+}\langle a\mid aaa=aa\rangle$ is prefix-automatic. Since $\widetilde{S}$ is a finite semigroup, we have $\widetilde{S}$ is  prefix-automatic.

$4)$ Suppose $|v|=3$. Then $S=sgp\langle A\mid u=v\rangle$ is biautomatic and prefix-automatic by Lemma \ref{3.2.3}.
\end{proof}

Theorem \ref{3.2.6} follows from Lemmas \ref{3.2.2} and \ref{3.2.5}.

\end{document}